\documentclass[11pt,leqno]{amsproc}

\usepackage{anysize}
\marginsize{3.5cm}{3.5cm}{2.5cm}{2.5cm}
\usepackage[]{hyperref}
\hypersetup{
    colorlinks=true,       % false: boxed links; true: colored links
    linkcolor=red,          % color of internal links
    citecolor=blue,        % color of links to bibliography
    filecolor=magenta,      % color of file links
   urlcolor=cyan           % color of external links
}

\usepackage{amsmath}
\usepackage{amsfonts,amssymb}
\usepackage{enumerate}
\usepackage{mathrsfs}
\usepackage{amsthm}

\def\lDx#1{\langle D_x\rangle^{#1}\,}

\theoremstyle{plain}
\newtheorem{theo}{Theorem}[section]
\newtheorem{prop}[theo]{Proposition}
\newtheorem{lemm}[theo]{Lemma}

\newtheorem{defi}[theo]{Definition}
\theoremstyle{definition}
\newtheorem{rema}[theo]{Remark}
\newtheorem{nota}[theo]{Notation}
\DeclareMathOperator{\cn}{div}

\DeclareSymbolFont{pletters}{OT1}{cmr}{m}{sl}
\DeclareMathSymbol{s}{\mathalpha}{pletters}{`s}

\def\B{B }

\def\defn{\mathrel{:=}}

\def\eps{\varepsilon}
\def\eval{\arrowvert_{y=\eta}}

\def\la{\left\vert}
\def\lA{\left\Vert}
\def\le{\leq}
\def\les{\lesssim}
\def\leo{}

\def\ma{a}
\def\mez{\frac{1}{2}}
\def\partialx{\nabla}

\def\ra{\right\vert}
\def\rA{\right\Vert}

\def\tdm{\frac{3}{2}}

\def\xC{\mathbf{C}}
\def\xN{\mathbf{N}}
\def\xR{\mathbf{R}}

\def\xZ{\mathbf{Z}}
\def\cF{ \mathcal{F}}

\newcommand{\hk}{\hspace*{.15in}}
\numberwithin{equation}{section}

\title{Hadamard well-posedness of the gravity water waves system}
\author{
Quang Huy Nguyen
\address{Quang Huy Nguyen. Laboratoire de Math\'ematiques d'Orsay, UMR 8628 du CNRS, Universit\'e Paris-Sud, 91405 Orsay Cedex, France}
}
\begin{document}

%\date{\empty}

\begin{abstract}
We consider in this article the system of (pure) gravity water waves in any dimension and in fluid domains with general bottoms. The unique solvability of the problem was established by Alazard-Burq-Zuily [Invent. Math, 198 (2014), no. 1, 71--163] at a low regularity level where the initial surface is $C^{\frac{3}{2}+}$ in terms of Sobolev embeddings, which allows the existence of free surfaces with unbounded curvature. Our result states  that the solutions obtained above depend continuously on initial data in the strong topology where they are constructed. This completes a well-posedness result in the sense of Hadamard.
\end{abstract}
\thanks{The author was supported in part by Agence Nationale de la Recherche
  project  ANA\'E ANR-13-BS01-0010-03.}
\maketitle
%%%%%%%%%%%%%%%%%%%%%%%%%

\section{Introduction}
\quad We are concerned in this paper with the local well-posedness theory for pure gravity water waves in Sobolev spaces. This problem has been considered by a great number of works, for instance  Nalimov \cite{Nalimov}, Yosihara \cite{Yosihara}, Craig \cite{Craig1985},  Wu \cite{WuInvent, WuJAMS}, Lannes \cite{LannesJAMS}, Lindblad \cite{LindbladAnnals}, Shatah-Zeng \cite{SZ3}, etc. Of course, in local theory, one of the main questions is to study well-posedness for data at optimal regularity, which is interesting in understanding the possible emergence of singularities.  This problem was recently studied by Alazard-Burq-Zuily \cite{ABZ, ABZ', ABZ4, ABZ1} and Hunter-Ifrim-Tataru \cite{HuIfTa}.  Under the Eulerian formulation, for fluid domains with general bottoms (see the precise assumptions in paragraph \ref{assu:domain} below), it was established in \cite{ABZ} the local well-posedness  in any dimension where the initial surface is
\begin{equation}\label{reg:ABZ}
\eta_0\in H^{s+\frac{1}{2}}(\xR^d)\quad s>1+\frac{d}{2}
\end{equation}
and the trace of the velocity field on the surface is in $H^{s}\hookrightarrow W^{1, \infty}$, hence is Lipschitz.  In particular, the free surface may have unbounded curvature. Note that the threshold \eqref{reg:ABZ} is $\mez$ above the scaling critical index $s_c=\mez+\frac{d}{2}$. Working in the holomorphic coordinates when $d=1$, the authors proved in \cite{HuIfTa} a well-posedness theory for initial surface $\eta_0\in H^2(\xR)$, which slightly improves the threshold \eqref{reg:ABZ} for 2D waves in the infinite depth case. Remark that however, the authors did not establish in \cite{HuIfTa} a direct uniqueness result at the end-point regularity $H^2(\xR)$. Uniqueness was proved in $H^3(\xR)$ and rough solutions in $H^2(\xR)$ were constructed as unique limits of smooth solutions in $H^3(\xR)$. Subsequently, in \cite{ABZ4} Alazard-Burq-Zuily took into account the dispersive property of water waves to prove Strichartz-type estimates and then improved a little the thresholds in \cite{ABZ} and \cite{HuIfTa} to non-Lipschitz initial velocity fields.\\
\hk  An important point in the Hadamard well-posedness is the continuity of the solution map  in the strong topology where the solution is constructed. For water waves, this fact is usually overlooked and widely believed to be true since it is rather standard in the context of quasilinear waves.  In the present paper, we would like to address this question, which turns out to be nontrivial especially for solutions at low regularities. Remark that the continuity of the solution map was indeed showed in \cite{HuIfTa} for 2D waves ($d=1$). Our main result states that the solution map is continuous at the level of regularity \eqref{reg:ABZ} in any dimension.  Let us mention also as a partial ill-posedness result, Chen-Marzuola-Spirn-Wright proved in \cite{CMSW} that the flow map is not $C^3$ if the free surface is not $C^{\frac{5}{2}+}$, in the case where surface tension effect is taken into account.
\subsection{Assumptions on the fluid domain}\label{assu:domain} We shall use the setting in \cite{ABZ} which is recalled here for reader's convenience. We work in a time-dependent fluid domain~$\Omega$
located underneath
a free surface~$\Sigma$ described by the unknown function $\eta$ and moving
in a fixed container denoted by~$\mathcal{O}$. More precisely,
$$
\Omega=\left\{\,(t,x,y)\in [0,T]\times{\mathbf{R}}^d\times\xR \, : \, (x,y) \in \Omega(t)\right\},
$$
where
$$
\Omega(t)=\left\{ (x,y)\in \mathcal{O} \, :\, y < \eta(t,x)\right\}.
$$
Assume that the fluid domain contains a fixed strip around the free surface, i.e, there exists~$h>0$ such that, for all~$t\in [0,T]$,
\begin{equation}\label{hypt}
\Omega_{h}(t)\defn \left\{ (x,y)\in {\mathbf{R}}^d\times \xR \, :\, \eta(t,x)-h < y < \eta(t,x)\right\} \subset \Omega(t).
\end{equation}
We also assume that the container~$\mathcal{O}$ (and hence the domain~$\Omega(t)$)
is connected but no regularity assumption is made on the bottom~$\Gamma\defn \partial\mathcal{O}$. In particular, $\Gamma$ can be empty, corresponding to the infinite depth case.
\subsection{The equations}
The Eulerian velocity $v: \Omega \to \xR^{d+1}$ is determined by the incompressible
and irrotational Euler equation and three boundary conditions described by the following system
\begin{equation}\label{wwi}
\left\{
\begin{aligned}
\partial_t v + (v \cdot \nabla_{x,y})v + \nabla_{x,y}P = -g e_y, \quad & \text{in}~ \Omega\\
\text{div}_{x,y} \, v =0, \quad  \quad \text{curl}_{x,y}\, v =0 \quad   &\text{in}~ \Omega\\
\partial_t \eta = \sqrt{1+ \vert \nabla_x \eta \vert^2} (v \cdot n) \quad & \text{on}~ \Sigma\\
P = 0 \quad & \text{on}~\Sigma\\
v \cdot \nu =0 \quad &\text{on}~\Gamma,
\end{aligned}
\right.
\end{equation}
where $g>0$ is the acceleration of gravity, $e_y$ is the unite vector $(x=0,y=1)$ and $P$ is the pressure. \\
There exists a velocity potential $\phi: \Omega \to \xR$
such that $v = \nabla_{x,y} \phi, $ thus $\Delta_{x,y} \phi = 0$ in $\Omega$. We introduce the trace of the potential on the surface
$$
\psi(t,x)= \phi(t,x,\eta(t,x))
$$
and the Dirichlet-Neumann operator
\begin{align*}
G(\eta) \psi &= \sqrt{1 + \vert \nabla_x \eta \vert ^2}
\Big( \frac{\partial \phi}{\partial n} \Big \arrowvert_{\Sigma}\Big)\\
&= (\partial_y \phi)(t,x,\eta(t,x)) - \nabla_x \eta(t,x) \cdot(\nabla_x \phi)(t,x,\eta(t,x)).
\end{align*}
Then (see \cite{CrSu}) the water waves system (\ref{wwi}) can be written in the Zakharov/Craig--Sulem formulation as a system of  $(\eta,\psi)$ 
\begin{equation}\label{ww}
\left\{
\begin{aligned}
\partial_t \eta &= G(\eta) \psi,\\
\partial_t \psi &= - \mez \vert \nabla_x \psi \vert^2 + \mez \frac{(\nabla_x \eta \cdot \nabla_x \psi + G(\eta)\psi)^2}{1+ \vert \nabla_x \eta \vert^2} - g \eta.
\end{aligned}
\right.
\end{equation}
Following \cite{ABZ} we shall consider the vertical and horizontal components of the velocity on the surface $\Sigma$  as unknowns which can be expressed in terms of $\eta$ and $\psi$ as
\begin{equation}\label{def:BV}
B = (v_y)\arrowvert_\Sigma = \frac{ \nabla_x \eta \cdot \nabla_x \psi + G(\eta)\psi} {1+ \vert \nabla_x \eta \vert^2},\quad V= (v_x)\arrowvert_\Sigma  =\nabla_x \psi - B \nabla_x \eta.
 \end{equation}
 Finally, recall that the Taylor coefficient 
\begin{equation}\label{def:Taylor}
a := -\frac{\partial P}{\partial y}\big\arrowvert_\Sigma
\end{equation}
 can be defined in terms of $\eta,\psi,B,V$ only (see \S 4.2 in \cite{ABZ} and \S 4.3.1 in \cite{LannesLivre}).
\subsection{Main result}
\begin{nota}Denote for $s\in \xR$,
$$Z^s=H^{s+\mez}(\xR^d)\times H^{s+\mez}(\xR^d)\times H^s(\xR^d)\times H^s(\xR^d). $$
\end{nota}
Let us recall first the local existence result proved in Theorem $1.2$, \cite{ABZ}.\\
Let $s>1+\frac{d}{2}$ and $(\eta^0,\psi^0)$ be such that
\begin{enumerate}
\item[H1:]
$ (\eta^0, \psi^0,  V^0 , B^0)\in Z^s,  $
\item[H2:] there exists $h>0$ such that $\eqref{hypt}$ holds initially,
\item[H3:] there exists  $c>0$ such that  the Taylor coefficient $a$ defined in \eqref{def:Taylor} verifying $a(0, x) \geq c,~\forall x \in \xR^d$.
\end{enumerate}
Then there exists $T>0$ such that the Cauchy problem for  system \eqref{ww} with initial data $(\eta^0, \psi^0)$  has a unique solution
$$(\eta, \psi) \in C^0([0,T], H^{s+\mez}(\xR^d) \times H^{s+\mez}(\xR^d)\big),
$$
such that
\begin{enumerate}
\item $(V,B) \in C^0\big([0,T], H^{s }(\xR^d) \times H^{s }(\xR^d)\big)$,
\item condition \eqref{hypt} holds with $h$ replaced by $h/2$ on $[0, T]$,
\item $a(t,x) \geq c/2,~\forall (t,x) \in [0,T]\times \xR^d$.
\end{enumerate}
\begin{rema}
The preceding Cauchy theory was obtained in \cite{ABZ} by proving a priori estimates, then contraction estimates for the solutions and finally, concluding by the standard method of regularizing initial data. More precisely, the contraction estimate in Theorem $5.1$, \cite{ABZ} shows that the solution map is Lipschitz continuous in the $Z^{s-1}$-topology, on bounded sets of $Z^s$. Then by interpolation, this implies the continuity of the solution map in $Z^{s'}$ (still on bounded sets of $Z^s$) for any $s'<s$ and not in $Z^s$, a priori (see \eqref{eq.lipschitz} and the argument following). The Hadamard well-posedness requires however such a continuity in the strong topology $Z^s$ where the solutions are constructed and this is our main result:
\end{rema}
\begin{theo}
\label{maintheo}
Let~$d\ge 1$,~$s>1+\frac d2$ and consider~$(\eta_{n}^0,\psi_{n}^0),~n\ge 0$ satisfying $(H1), (H2), (H3)$ uniformly in $n$ and
\begin{equation}\label{hyp}
(\eta_n^0, \psi_n^0, V_n^0, B^0_n)_{n\ge 1}~\text{converges to}~ (\eta^0_0, \psi^0_0, V^0_0, B^0_0)~\text {in}~Z^s.
\end{equation}
Then there exists~$T>0$  independent of $n\ge 0$ such that
the Cauchy problem for \eqref{ww}
with initial data $(\eta^{0}_n,\psi^{0}_n)$ has a unique solution
$$
(\eta_n,\psi_n)\in C^0\big([0,T], H^{s+\mez}({\mathbf{R}}^d)\times H^{s+\mez}({\mathbf{R}}^d)\big)
$$
and
\begin{equation}\label{conclusion}
(\eta_n, \psi_n, V_n, B_n)_{n\ge 1}~\text{converges to}~ (\eta_0, \psi_0, V_0, B_0)~\text {in}~C^0([0, T], Z^s).
\end{equation}
\end{theo}
\begin{rema}
To study water waves in the case that the free surfaces are neither periodic nor decaying to zero at infinity, it was established in \cite{ABZ'} a similar Cauchy theory to the one in \cite{ABZ}, in the framework of  uniformly local Sobolev spaces (Kato sapces) $H^s_{ul}$ (see Definition $2.1$, \cite{ABZ'}) except that the obtained solution is
\[
(\eta, \psi, B, V)\in L^\infty([0, T], Z^s_{ul})\cap C^0([0, T],  Z^{s'}_{ul}),\quad\forall s'<s
\]
here, we have used the obvious notation $Z^s_{ul}$ for the uniformly local version of $Z^s$.\\
 Remark that using the method in the present paper, modulo some additional commutator estimates due to the presence of a cut-off function appearing in the definition of Kato spaces, we can obtain another version of Theorem \ref{maintheo} for this setting. As a consequence, it is easy to prove that the solution is actually continuous in time with the strong topology in space:
\[
(\eta, \psi, B, V)\in C^0([0, T],  Z^{s}_{ul}).
\]
\end{rema}
%%%%%%%%%%%%%%%
\subsection{On the proof of the main result}\label{idea}
System \eqref{ww} was reduced to the following single equation in \cite{ABZ}
\begin{equation}\label{formepara}
\partial_t u+T_V\cdot \nabla u +iT_\gamma u=F(U)
\end{equation}
where $V$ is defined in \eqref{def:BV}, $\gamma$ is a symbol of order $1/2$,~ $T_V,~T_\gamma$ are paradifferential operators (see section \ref{para:review} below); $U=(\eta, \psi, B, V)$ is the original unknown and $u$ is the new unknown obtained after performing paralinearization and symmetrization. A direct application of the well-known  Bona-Smith argument \cite{Bona} (see also \cite{Tzvetkov} for an illustration of this argument for the Burgers equation) does not seem to work in our case. According to this argument, one regularizes the data and makes use of the convergence of the corresponding regularized solutions in strong norm together with nice bounds for solution in terms of initial data. However, we do not have these in hand from \cite{ABZ} where only weak convergence is hoped to hold. We shall use an adjusted argument suggested in \cite{ABZ1}, which can be sketched as follows (to simplify the argument, let us pretend as if $U=u\in Z^s$ in \eqref{formepara}).
\begin{enumerate}
 \item Regularizing the solution $u$ itself by $K_\eps u$, where  $K_\eps$ is a multiplier cutting away the frequencies less than $1/\eps,~\eps\in (0, 1)$. Then $K_\eps u$ solves
\[
\big(\partial_t+T_V\cdot \nabla  +iT_\gamma )(K_\eps u)=K_\eps F(u)+G_\eps(u)
\]
where $G_\eps(u)$ is comprised of the commutators of $K_\eps$ with $T_V\cdot \nabla$ and $T_\gamma$.
\item Proving that, roughly speaking, the commutator of $K_\eps$ with the nonlinear function $F$ leaves a small error
\begin{equation}\label{intro:KF}
\lA K_\eps F(u)\rA_{Z^s}\le C\big( o(1)_{\eps\to 0}+\lA K_\eps u\rA_{Z^s}\big)
\end{equation}
where $C>0$ depending only on $\|u\|_{Z^s}$; similarly for $G_\eps u$. Deriving then the energy estimate
\begin{equation}\label{intro:energy}
\sup_{t\in [0, T]}\lA K_\eps u(t)\rA_{Z^s}\le C' \big(o(1)_{\eps\to 0}+\lA K_\eps u(0)\rA_{Z^s}\big),
\end{equation}
where again, $C'>0$ depending only on $\|u\|_{Z^s}$.
\item  Assume that there is a sequence $u_n(0),~n\ge 1$ converging to $u_0(0)$ in $H^s$, we estimate
\begin{align*}
&\lA u_n-u_0\rA_{C^0([0, T], Z^s)}\\\nonumber
&\le \lA (I-K_{\eps})(u_n-u_0)\rA_{C^0([0, T], Z^s)} +\lA K_{\eps}(u_n-u_0)\rA_{C^0([0, T], Z^s)}\\\nonumber
&\le \frac{C}{\eps}\lA (u_n-u_0)\rA_{C^0([0, T], Z^{s-1})}+\lA K_{\eps}u_0\rA_{C^0([0, T], Z^s)}+ \lA K_{\eps}u_n\rA_{C^0([0, T], Z^s)}.
\end{align*}
 Making use of the {\it continuity of the solution map in weak norms} in \cite{ABZ} yields
\[
\lim_{n\to\infty}\lA (u_n-u_0)\rA_{C^0([0, T], Z^{s-1})}=0.
\]
Finally, the energy estimate \eqref{intro:energy} applied with $u=u_n,~n\ge 0$ implies that
\[
\lim_{\eps\to 0}\lA K_{\eps}u_n\rA_{C^0([0, T], Z^s)}=0,~\text{uniformly in }~n\ge 0,
\]
from which one concludes $\lim_{n\to\infty}\lA u_n-u_0\rA_{C^0([0, T], Z^s)}=0$.
\end{enumerate}
The continuity of the solution map in weak norms was established in \cite{ABZ}, Theorem $5.1$. The main step will be deriving \eqref{intro:KF}. To obtain this, the idea consists in revisiting all the estimates in deriving the reduction \eqref{formepara} in \cite{ABZ} and show that in the high frequency regime, the estimates involving the highest norms always appear linearly.\\ %This implies the following fact about the reduced equation \eqref{formepara}: the high frequency part of $F(U)$ can be controlled by the high frequency part of $U$, modulo a small error (see Proposition \ref{prop:sym} below).
Finally, let us mention that in \cite{HuIfTa}ề the authors  proposed another method, namely, {\it frequency envelopes} to establish the continuous dependence of data-solution.
\section{A review of Paradifferential Calculus}\label{para:review}
\begin{defi}
1. (Littlewood-Paley decomposition) Let~$\kappa\in C^\infty_0({\mathbf{R}}^d)$ be such that
$$
\kappa(\theta)=1\quad \text{for }\la \theta\ra\le 1.1,\qquad
\kappa(\theta)=0\quad \text{for }\la\theta\ra\ge 1.9.
$$
Define
\begin{equation*}
\kappa_k(\theta)=\kappa(2^{-k}\theta)\quad\text{for }k\in \xZ,
\qquad \varphi_0=\kappa_0,\quad\text{ and }
\quad \varphi_k=\kappa_k-\kappa_{k-1} \quad\text{for }k\ge 1.
\end{equation*}
Given a temperate distribution $u$ and an integer $k$ in $\xN$, we  introduce 
\[
S_k u=\kappa_k(D_x)u,\quad \Delta_0u=S_0u,\quad \Delta_k u=S_k u-S_{k-1}u\quad \text{for}~k\ge 1.
\]
 Then we have the formal dyadic partition of the unity
$$
u=\sum_{k=0}^{\infty}\Delta_k u.
$$
2. (Zygmund spaces) For any real number $s$, the Zygmund space~$C^{s}_*({\mathbf{R}}^d)$ is defined as the
space of all the tempered distributions~$u$ satisfying
$$
\lA u\rA_{C^{s}_*}\defn \sup_q 2^{qs}\lA \Delta_q u\rA_{L^\infty}<+\infty.
$$
3. (H\"older spaces) For~$k\in\xN$, we denote by $W^{k,\infty}({\mathbf{R}}^d)$ the usual Sobolev spaces.
For $\rho= k + \sigma$ with $k\in \xN$ and $\sigma \in (0,1)$, $W^{\rho,\infty}({\mathbf{R}}^d)$
denotes the space of all function $u\in W^{k, \infty}(\xR^d)$ such that all the $k^{th}$ derivatives of $u$ are  $\sigma$-H\"older continuous on $\xR^d$.
\end{defi}
Next, we review the notations and basic results of the Bony paradifferential calculus (see \cite{Bony,MePise}). Here, we follow the presentation of M\'etivier in \cite{MePise} (see also \cite{ABZ}).
\begin{defi}
1. (Symbols) Given~$\rho\in [0, \infty)$ and~$m\in\xR$,~$\Gamma_{\rho}^{m}({\mathbf{R}}^d)$ denotes the space of
locally bounded functions~$a(x,\xi)$
on~${\mathbf{R}}^d\times({\mathbf{R}}^d\setminus 0)$,
which are~$C^\infty$ with respect to~$\xi$ for~$\xi\neq 0$ and
such that, for all~$\alpha\in\xN^d$ and all~$\xi\neq 0$, the function
$x\mapsto \partial_\xi^\alpha a(x,\xi)$ belongs to~$W^{\rho,\infty}({\mathbf{R}}^d)$ and there exists a constant
$C_\alpha$ such that,
\begin{equation*}%\label{para:10}
\forall\la \xi\ra\ge \mez,\quad
\lA \partial_\xi^\alpha a(\cdot,\xi)\rA_{W^{\rho,\infty}(\xR^d)}\le C_\alpha
(1+\la\xi\ra)^{m-\la\alpha\ra}.
\end{equation*}
Let $a\in \Gamma_{\rho}^{m}({\mathbf{R}}^d)$, we define the semi-norm
\begin{equation}\label{defi:norms}
M_{\rho}^{m}(a)=
\sup_{\la\alpha\ra\le 2(d+2) +\rho ~}\sup_{\la\xi\ra \ge 1/2~}
\lA (1+\la\xi\ra)^{\la\alpha\ra-m}\partial_\xi^\alpha a(\cdot,\xi)\rA_{W^{\rho,\infty}({\mathbf{R}}^d)}.
\end{equation}
2. (Paradifferential operators) Given a symbol~$a$, we define
the paradifferential operator~$T_a$ by
\begin{equation}\label{eq.para}
\widehat{T_a u}(\xi)=(2\pi)^{-d}\int \chi(\xi-\eta,\eta)\widehat{a}(\xi-\eta,\eta)\psi(\eta)\widehat{u}(\eta)
\, d\eta,
\end{equation}
where
$\widehat{a}(\theta,\xi)=\int e^{-ix\cdot\theta}a(x,\xi)\, dx$
is the Fourier transform of~$a$ with respect to the first variable;
$\chi$ and~$\psi$ are two fixed~$C^\infty$ functions such that:
\begin{equation}\label{cond.psi}
\psi(\eta)=0\quad \text{for } \la\eta\ra\le \frac{1}{5},\qquad
\psi(\eta)=1\quad \text{for }\la\eta\ra\geq \frac{1}{4},
\end{equation}
and~$\chi(\theta,\eta)$  is defined by
$
\chi(\theta,\eta)=\sum_{k=0}^{+\infty} \kappa_{k-3}(\theta) \varphi_k(\eta).
$
\end{defi}
\begin{defi}\label{defi:order}
Let~$m\in\xR$.
An operator~$T$ is said to be of  order~$\leo m$ if, for all~$\mu\in\xR$,
it is bounded from~$H^{\mu}$ to~$H^{\mu-m}$.
\end{defi}
Symbolic calculus for paradifferential operators is summarized in the following theorem.
\begin{theo}\label{theo:sc}(Symbolic calculus)
Let~$m\in\xR$ and~$\rho\in [0, \infty)$. \\
$(i)$ If~$a \in \Gamma^m_0({\mathbf{R}}^d)$, then~$T_a$ is of order~$\leo m$.
Moreover, for all~$\mu\in\xR$ there exists a constant~$K$ such that
\begin{equation}\label{esti:quant1}\lA T_a \rA_{H^{\mu}\rightarrow H^{\mu-m}}\le K M_{0}^{m}(a).
\end{equation}
$(ii)$ If~$a\in \Gamma^{m}_{\rho}({\mathbf{R}}^d), b\in \Gamma^{m'}_{\rho}({\mathbf{R}}^d)$ then
$T_a T_b -T_{a\sharp  b}$ is of order~$\leo m+m'-\rho$ where
\[
a\sharp b:=\sum_{|\alpha|<\rho}\frac{(-i)^{\alpha}}{\alpha !}\partial_{\xi}^{\alpha}a(x, \xi)\partial_x^{\alpha}b(x, \xi).
\]
Moreover, for all~$\mu\in\xR$ there exists a constant~$K$ such that
\begin{equation}\label{esti:quant2}
\lA T_a T_b  - T_{a \sharp b}   \rA_{H^{\mu}\rightarrow H^{\mu-m-m'+\rho}}%&
\le
K M_{\rho}^{m}(a)M_{0}^{m'}(b)+K M_{0}^{m}(a)M_{\rho}^{m'}(b).
\end{equation}
$(iii)$ Let~$a\in \Gamma^{m}_{\rho}({\mathbf{R}}^d)$ with $\rho \in [0, 1]$. Denote by
$(T_a)^*$ the adjoint operator of~$T_a$ and by~$\overline{a}$ the complex conjugate of~$a$. Then
$(T_a)^* -T_{\overline{a}}$ is of order~$\leo m-\rho$.
Moreover, for all~$\mu\in \xR$, there exists a constant~$K$ such that
\begin{equation}\label{esti:quant3}
\lA (T_a)^*   - T_{\overline{a}}   \rA_{H^{\mu}\rightarrow H^{\mu-m+\rho}}\le
K M_{\rho}^{m}(a).
\end{equation}
\end{theo}
\begin{nota}
Given two functions~$a,~u$ defined on~$\xR^d$ we denote the remainder of the Bony decomposition by
$$
R(a,u)=au-T_a u-T_u a.
$$
\end{nota}
We shall use frequently various estimates about paraproducts (see Chapter $2$ in~\cite{BCD} and Section $2$ in \cite{ABZ}) which are recalled here.
\begin{theo}\label{pproduct}
\begin{enumerate}[i)]
\item  Let~$\alpha,\beta\in \xR$. If~$\alpha+\beta>0$ then
\begin{align}
&\lA R(a,u) \rA _{H^{\alpha + \beta-\frac{d}{2}}}
\leq K \lA a \rA _{H^{\alpha}}\lA u\rA _{H^{\beta}},\label{Bony} \\
&\lA R(a,u) \rA _{H^{\alpha + \beta}} \leq K \lA a \rA _{C^{\alpha}_*}\lA u\rA _{H^{\beta}}.\label{Bony3}
\end{align}
\item Let $s_0, s_1, s_2\in \xR$. If $s_0\le s_2$ and~$s_0 < s_1 +s_2 -\frac{d}{2}$ then there exists a constant~$K$ such that
\begin{equation}\label{boundpara}
\lA T_a u\rA_{H^{s_0}}\le K \lA a\rA_{H^{s_1}}\lA u\rA_{H^{s_2}}.
\end{equation}
\item\label{parane}  Let~$m>0$ and~$s\in \xR$. Then there exists a constant~$K$ such that
\begin{equation}
\lA T_a u\rA_{H^{s-m}}\le K \lA a\rA_{C^{-m}_*}\lA u\rA_{H^{s}}.\label{niS}%\\
\end{equation}
\item \label{it.1} Let $s_0, s_1, s_2\in \xR$. If $s_0\le s_1,~s_0\le s_2,~ s_1+s_2>0,~s_0< s_1+s_2-\frac{d}{2}$ then there exists a constant~$K$ such that
\begin{equation}\label{pr}
\lA u_1 u_2 \rA_{H^{s_0}}\le K \lA u_1\rA_{H^{s_1}}\lA u_2\rA_{H^{s_2}}.
\end{equation}
\item Let~$r,\mu, \gamma\in \xR$. If $\gamma\le r,~r+\mu>0,~\gamma < r+\mu-\frac{d}{2}$, then there exists a constant~$K$ such that
$$
\lA au - T_a u\rA_{H^{\gamma}}\le K \lA a\rA_{H^{r}}\lA u\rA_{H^\mu}.
$$
\item \label{it.6} Let~$s>\frac{d}{2}$ and consider~$F\in C^\infty(\xC^N)$ satisfying~$F(0)=0$.
Then there exists a non-decreasing function~$\mathcal{F}\colon\xR_+\rightarrow\xR_+$
such that, for any~$U\in H^s({\mathbf{R}}^d)^N$,
\begin{equation}\label{esti:F(u)}
\lA F(U)\rA_{H^s}\le \mathcal{F}\bigl(\lA U\rA_{L^\infty}\bigr)\lA U\rA_{H^s}.
\end{equation}
\end{enumerate}
\end{theo}
\begin{theo}\label{paralin}\protect{\cite[Theorem~2.92]{BCD}}(Paralinearization)
Let $F$ be a $C^{\infty}$ function on $\xR$ such that $F(0)=0$. If $u\in H^r(\xR^d)$ with $r>\frac{d}{2}$, then
\[
F(u)-T_{F'(u)}u\in H^{2r-\frac{d}{2}}.
\]
Moreover, if $\rho\in (0, r-d/2)$ and $\rho\notin\xN$ then
\[
\lA F(u)-T_{F'(u)}u\rA_{ H^{r+\rho}}\le C(\lA u\rA_{L^{\infty}})\lA u\rA_{C_*^{\rho}})\lA u\rA_{H^r}.
\]
\end{theo}
%%%%%%%%%%%%%%%%%%%
\section{Proof of Theorem \ref{maintheo}}\label{main}
Let $(\eta_n, \psi_n, B_n, V_n),~n\ge 0$ be a sequence of solutions to system \eqref{ww} on the time interval $[0, T_n]$. Hereafter, we fix an index $s>\frac{3}{2}+\frac{d}{2}$ and fix
\begin{equation}\label{rhoi}
0<\rho_0<\frac{1}{2}\min\{1, s-1-\frac{d}{2}\}.
\end{equation}
By Proposition 4.1, \cite{ABZ} we have the following {\it a priori} estimate
\begin{equation}\label{apriori}
\mathcal{M}^{(n)}_s(T)\le \mathcal{F}\bigl(\mathcal{F}(\mathcal{M}^{(n)}_{s,0})+T\mathcal{F}\bigl(\mathcal{M}^{(n)}_s(T)%+Z_r(T)
\bigr)\bigr),\quad\forall T\le T_n
\end{equation}
where $\cF:\xR^+\to \xR^+$ is a nondecreasing function and
\begin{equation}\label{defi:MsZr}
\begin{aligned}
\mathcal{M}_s^{(n)}(T)&\defn \sup_{\tau\in [0,T]}
\lA (\eta_n(\tau)), \psi_n(\tau),B_n(\tau),V_n(\tau)\rA_{H^{s+\mez}\times H^{s+\mez} \times H^s\times H^s},\\
\mathcal{M}^{(n)}_{s,0}&\defn \lA (\eta_n(0), \psi_n(0),B_n(0),V_n(0))\rA_{H^{s+\mez}\times H^{s+\mez}\times H^s\times H^s}.
\end{aligned}
\end{equation}
Since $\mathcal{M}^{(n)}_{s,0}$ is bounded, the {\it a priori} estimate \eqref{apriori} gives us the uniform existence of $T$ as claimed Theorem \ref{maintheo} and moreover $\mathcal{M}^{(n)}_s(T)$ is bounded uniformly in $n$. The rest of this paper is devoted to prove conclusion \eqref{conclusion}. Denote
$$U_n=(\eta_n, \psi_n, B_n, V_n),~n\ge 0$$
and recall also that $Z^s=H^{s+\mez}\times H^s\times H^s\times H^{s+\mez}$. Our goal is to prove that $U_n$ converges to $U_0$ in $C^0([0, T], Z^s)$. Notice that by Theorem 5.1, \cite{ABZ} and the boundedness of $M^{(n)}_s$ we have
\begin{multline}\label{eq.lipschitz}
\|(\eta_n-\eta_0, \psi_n-\psi_0, B_n-B_0, V_n-V_0)\|_{C^0([0, T], Z^{s-1})} \\
\leq C\|(\eta_n-\eta_0, \psi_n-\psi_0, B_n-B_0, V_n-V_0)\mid_{t=0}\|_{Z^{s-1}}.
\end{multline}
By interpolating the preceding contraction estimate with the obvious estimate
\[
\|(\eta_n-\eta_0, \psi_n-\psi_0, B_n-B_0, V_n-V_0)\|_{C^0([0, T], Z^{s})}\le C
\]
 for some $C$ independent of $n$, one deduces that $U_n$ converges to $U_0$ in $C^0([0, T], Z^r)$ for any $r<s$. Our goal  is to prove  that  in fact, this convergence still holds in the strong topology of the solution, i.e, in $C^0([0, T], Z^s)$. \\
Let $J_\eps$ denote the usual Friedrichs mollifiers, defined by
$J_\eps =j_\eps(D_x):= \jmath(\eps D_x)$ where~$\jmath \in C^\infty_0({\mathbf{R}}^d)$,~$0\le \jmath\le 1$ satisfying
$$
\jmath(\xi)=1\quad \text{for }\la \xi\ra\le 1,\quad \jmath(-\xi)=\jmath(\xi),\quad
\jmath(\xi)=0\quad\text{for }\la \xi\ra\ge 2.
$$
We set for $\eps >0$ small
\[
K_{\eps}=I-J_{\eps},\quad k_{\eps}=1-\jmath_{\eps}
\]
Since $I=K_{\eps}^2+2J_{\eps}-J_{\eps}^2$  we write (guided by \cite{ABZ1})
\begin{align}\label{est:main}
&\lA U_n-U_0\rA_{C^0([0, T], Z^s)}\\\nonumber
&\le 2\lA J_{\eps}(U_n-U_0)\rA_{C^0([0, T], Z^s)}+\lA J^2_{\eps}(U_n-U_0)\rA_{C^0([0, T], Z^s)} +\lA K_{\eps}^2(U_n-U_0)\rA_{C^0([0, T], Z^s)}\\\nonumber
&\le C\lA J_{\eps}(U_n-U_0)\rA_{C^0([0, T], Z^s)}+\lA K^2_{\eps}U_0\rA_{C^0([0, T], Z^s)}+ \lA K^2_{\eps}U_n\rA_{C^0([0, T], Z^s)}.
\end{align}
For fixed $\eps$, 
\begin{equation}\label{contraction:low}
\lA J_\eps(U_n-U_0)\rA_{C^0([0, T], Z^{s-1})}\le C\frac{1}{\eps}\lA (U_n-U_0)\rA_{C^0([0, T], Z^{s-1})}
\end{equation}
 tending to $0$ by means of the contraction estimate \eqref{eq.lipschitz}. The proof is complete if one can show that
\begin{equation}\label{est1}
\lim_{\eps\to 0} \lA K^2_{\eps}U_n\rA_{C^0([0, T], Z^s)}=0 \quad\text{uniformly with respect to}~ n.
\end{equation}
By Lemma \ref{Keps} $(i)$ below, the limit in \eqref{est1} holds for each $n$. However, to prove that it holds uniformly in $n$, as in the estimate \eqref{intro:energy} of the general strategy, we encounter a technical difficulty.  Let $o(1)_{\eps\to 0}$ denotes some constant of the form $\eps^\kappa$ with $\kappa>0$ independent of the solution. We shall prove that there exists a constant $M>0$ independent of $n$ such that
\begin{multline}\label{est:main1}
 \lA K^2_{\eps}U_n\rA_{C^0([0, T], Z^s)}\le Mo(1)_{\eps\to 0}+M\lA K^2_{\eps}U_n(0)\rA_{Z^s}+M\lA K_{\eps}U_0\rA_{C([0, T], Z^s)}\cr
+ M\lA J^2_{\eps}(U_n-U_0)\rA_{C([0, T], Z^s)}
\end{multline}
The unexpected term $ \lA J^2_{\eps}(U_n-U_0)\rA_{C([0, T], Z^s)}$ can be treated as in \eqref{contraction:low}, hence one concludes the proof.  For a technical reason (see the estimate \eqref{technical:K2}), there is a loss of $K_\eps$ in the third term on the right-hand side of \eqref{est:main1}, this is why we need to deal with $K_\eps^2$ in \eqref{est:main} and not simply $K_\eps$ as in the general strategy \ref{idea}.
\begin{rema}\label{important}
 From here to  Lemma \ref{Keps:gamma}, all the results  that are stated for $K_{\eps}$ also hold with $K^2_{\eps}$.
\end{rema}
 We prove in the following some useful lemmas that will be used often in our proof.
\begin{lemm}\label{Keps}
$(i)$ Let $\mu\in \xR$ and $A$ is a compact set in $H^{\mu}(\xR^d)$. Then
\[
\lim_{\eps\to 0}\lA K_{\eps}u\rA_{H^{\mu}}=0,\quad\text{uniformly on}~ A.
\]
$(ii)$ There exists a constant $C$ such that  for all  $t>0, \mu\in \xR$ and $u\in H^{\mu}(\xR^d)$, there holds
\[
\lA K_{\eps}u\rA_{H^{\mu}}\le C\eps^t\lA u\rA_{H^{\mu+t}}.
\]
$(iii)$ Let $\mu,~\mu'\in\xR$, $a\in \Gamma^m_r$ with $r\le 1$. Suppose that $m-r<\mu'-\mu$. Then there exist $\theta>0$ and $C=C(\mu)>0$ such that
\[
\lA [K_{\eps}, T_a]u\rA_{H^\mu}\le C\eps^{\theta}M^m_r(a)\|u\|_{H^{\mu'}},\quad\forall u\in H^{\mu'}(\xR^d).
\]
\end{lemm}
\begin{proof}
The assertions $(i)$ and $(ii)$ were proved in Lemma 5, \cite{Bona}. We give the proof for $(iii)$. By Theorem \ref{theo:sc} $(ii)$, $[K_{\eps}, T_a]$ is of order $m-r$. Let us pick a real number $\rho$ such that
\[
0<\rho<\mu'-\mu+r-m.
\]
Then we have on the one hand,
\begin{equation}\label{Keps1}
\lA [K_{\eps}, T_a]u \rA_{H^{\mu+\rho}}\les M^m_r(a) \lA u\rA_{H^{\mu+\rho+m-r}} \les M^m_r(a) \lA u\rA_{H^{\mu'}}.
\end{equation}
On the other hand, for $t>0$ and $\alpha>0$
\begin{align*}
\lA [K_{\eps}, T_a]u \rA_{H^{\mu-t}}&\le
\lA K_{\eps}T_au \rA_{H^{\mu-t}}+\lA T_aK_{\eps}u \rA_{H^{\mu-t}}\\ \nonumber
&\lesssim \eps^{\alpha}\lA T_au \rA_{H^{\mu-t+\alpha}}+M^m_0(a)\lA K_{\eps}u \rA_{H^{\mu-t+m}}\\
&\lesssim \eps^{\alpha}M^m_0(a)\lA u \rA_{H^{\mu-t+m+\alpha}}+\eps^{\alpha}M^m_0(a)\lA u \rA_{H^{\mu-t+m+\alpha}}.
\end{align*}
Thus, if we choose $\alpha,~t>0$ such that $\mu-t+m+\alpha\le \mu'$ then
\begin{equation}\label{Keps2}
\lA [K_{\eps}, T_a]u \rA_{H^{\mu-t}}\lesssim \eps^{\alpha}M^m_0(a)\lA u\rA_{H^{\mu'}}.
\end{equation}
Combining \eqref{Keps1}, \eqref{Keps2} and the interpolation inequality in Sobolev norms we obtain the desired result.
\end{proof}
\begin{rema}
 Lemma \ref{Keps} $(iii)$ will be  frequently applied with $\mu=\mu'$ and $m<r\le 1$.
\end{rema}
\begin{lemm}\label{lemma:K2}
$(i)$ Let $s,~m\in \xR,~\alpha\in (0, 1]$ and $a\in \Gamma^m_{1+\alpha}(\xR^d)$. Then there exists $\theta>0$ such that for any $\mu\in \xR$, one can find $C_\mu>0$ such that
\[
\| [K^2_\eps\langle D_x\rangle^s, T_a]u\|_{H^\mu}\le C_\mu M^m_{1+\alpha}(a)\big(\eps^\theta \| u\|_{H^{\mu+s+m-1}}+\| K_\eps u\|_{H^{\mu+s+m-1}}\big),
\]
for all $u\in H^{s+m+\mu-1}(\xR^d)$.\\
$(ii)$ Let $s, m\in \xR$, $\alpha=\beta+\rho\in [0, 1]$ with $\beta, \rho\in [0, 1]$ and $a\in \Gamma^m_{\alpha}(\xR^d)$. Then  for any $\mu\in \xR$, one can find $C_\mu>0$ such that
\[
\| [K^2_\eps\langle D_x\rangle^s, T_a]u\|_{H^\mu}\le \eps^\rho C_\mu M^m_\alpha(a)\| u\|_{H^{\mu+s+m-\beta}},\quad \forall u\in H^{\mu+s+m-\beta}(\xR^d).
\]
\end{lemm}
\begin{proof}
It is easy to see
\[
[K^2_\eps\langle D_x\rangle^s, T_a]u=K_\eps[K_\eps \langle D_x\rangle^s, T_a]u+[K_\eps, T_a](\langle D_x\rangle^sK_\eps u).
\]
$(i)$ $a\in \Gamma^m_{1+\alpha}$. Since $k_\eps(\xi) \langle \xi\rangle^s\in \Gamma^{s}_{1+\alpha}$, Theorem \ref{theo:sc} $(ii)$ implies
\[
[K_\eps  \langle D_x\rangle^s, T_a]=R+T_c
\]
where $R$ is of order $s+m-1-\alpha$ and $c=-i\partial_\xi (k_\eps(\xi)\langle \xi\rangle^s)\partial_xa(x,\xi)\in \Gamma^{s+m-1}_{\alpha}$. Similarly,
\[
[K_\eps, T_a]=R'+T_{e}
\]
with $R'$ of order $m-1-\alpha$ and $e=-i\partial_\xi k_\eps(\xi)\partial_xa(x,\xi)\in \Gamma^{m-1}_\alpha$. Consequently, we get
\begin{align*}
[K_\eps, \langle D_x\rangle^sT_a]u&=K_\eps Ru+K_\eps T_cu+R' (\langle D_x\rangle^sK_\eps u)+T_{e}(\langle D_x\rangle^sK_\eps u)\\
&=K_\eps Ru+T_cK_\eps u+[T_c, K_\eps]u+R' (\langle D_x\rangle^sK_\eps u)+T_{e}(\langle D_x\rangle^sK_\eps u).
\end{align*}
By virtues of Lemma \ref{Keps} $(ii)$ and Theorem \ref{theo:sc} $(ii)$ there holds
\[
\lA K_\eps Ru\rA_{H^{\mu}}+\lA R'(\langle D_x\rangle^sK_\eps u)\rA_{H^{\mu}}  \les \eps^\alpha M^m_{1+\alpha}(a)\|  u\|_{H^{\mu+s+m-1}}.
\]
On the other hand, Lemma \ref{Keps} $(iii)$ (applied with $m:=s+m-1,~r=\alpha,~\mu'=\mu+s+m-1$) and Theorem \ref{theo:sc} $(ii)$ imply that for some $\theta_1>0$,
\[
\lA [K_\eps, T_c]u\rA_{H^{\mu}}\les \eps^{\theta_1} M^{s+m-1}_{\alpha}(c)\| u\|_{H^{\mu+s+m-1}}\les \eps^{\theta_1} M^m_{1+\alpha}(a)\| u\|_{H^{\mu+s+m-1}},
\]
from which we complete the proof of $(i)$.\\
$(ii)$ $a\in \Gamma^m_\alpha$. In this case
\[
[K_\eps\langle D_x\rangle^s, T_a]=R,\quad [K_\eps, T_a]=R'
\]
 with $R$ of order $s+m-\alpha$ and $R'$ of order $m-\alpha$. The result follows as above by using again Lemma \ref{Keps} $(ii)$ and Theorem \ref{theo:sc} $(ii)$.
\end{proof} 
\subsection{Paralinearization of the Dichlet-Neumann operator}
Let ~$\eta \in W^{1, \infty}({\mathbf{R}}^d)$  and~$ f\in H^{\mez}({\mathbf{R}}^d)$, we recall the construction of ~$G(\eta)f$ in \cite{ABZ} (see Section 3 there).
\subsubsection{Straightening the free boundary} We recall first the diffeomorphism that straightens the free boundary introduced in \cite{ABZ}. Recall that
\begin{equation}\label{lesomega}
\Omega_h = \{(x,y)\in \xR^d\times \xR: \eta(x)-h<y<\eta(x)\}.
 \end{equation}
 For $\delta>0$ sufficiently small, the map $(x,z) \mapsto  (x, \rho(x,z))$ from $\widetilde{\Omega}\defn\xR^d\times(-1,0)$ to $\Omega_h$ with
\begin{equation}\label{diffeo}
\rho(x,z):=  (1+z)e^{\delta z\langle D_x \rangle }\eta(x) -z \big\{e^{-(1+ z)\delta\langle D_x \rangle }\eta(x) -h\big\}
\end{equation}
is a Lipschitz-difffeomorphism.
\begin{nota}\label{tilde}
For any function $f$ defined on $\Omega$, we set
\begin{equation}\label{change}
 \widetilde{f}(x,z) = f(x,\rho(x,z)),~(x, z)\in \widetilde\Omega,
\end{equation}
the image of $f$ via the diffeomorphism $(x, z)\mapsto (x, \rho(x, z))$. Then
\begin{equation}\label{lambda}
 \left\{
 \begin{aligned}
  \frac{\partial f}{\partial y}(x,\rho(x,z)) &= \frac{1}{\partial_z \rho}\partial_z \widetilde{f}(x,z):=\Lambda_1\widetilde{f}(x,z) \\
  \nabla_x f(x,\rho(x,z)) &= \big(\nabla_x \tilde{f} - \frac{\nabla_x \rho}{\partial_z \rho}\partial_z \widetilde{f}\big)(x,z) := \Lambda_2 \widetilde{f}(x,z).
  \end{aligned}
  \right.
\end{equation}
\end{nota}
\subsubsection{Definition of the Dirichlet-Neumann operator}
For $\eta\in W^{1,\infty}(\xR^d)$ and $f\in H^\mez(\xR^d)$, there exists a unique variational solution $\phi$ to the boundary value problem
\begin{equation}\label{v}
\Delta_{x,y}\phi=0\text{ in }\Omega,\quad
\phi\arrowvert_{\Sigma}=f,\quad \partial_n \phi\arrowvert_{\Gamma}=0
\end{equation}
and $\phi$ satisfies
\begin{equation}\label{est:vari}
\int_\Omega |\nabla_{x,y}\phi|^2dxdy\le K\Vert f\Vert^2_{H^\mez(\xR^d)}
\end{equation}
for some constant $K$ depending only on the Lipschitz norm of $\eta$ (see Section $3.1$ in \cite{ABZ}).\\ 
Let $v(x,z):=\phi(x, \rho(x,z))$ be the image of $\phi$ via the diffeomorphism $(x, z)\mapsto (x, \rho(x,z))$. The Dirichlet-Neumann operator is given by the formula
\begin{equation*}
G(\eta) f= \frac{1 +|\partialx \rho |^2}{\partial_z \rho} \partial_z v  - \partialx \rho \cdot\partialx v
\big\arrowvert_{z=0}.
\end{equation*}
\begin{nota}
We will denote by $\cF$ any function from $\xR^+$ to $\xR^+$, nondecreasing in each argument and $\cF$ may change from line to line but is independent of relevant quantities. 
\end{nota}
Remark that
\begin{equation}\label{zeta12}
\zeta_1 \defn  \frac{1 +|\partialx \rho |^2}{\partial_z \rho},\quad \zeta_2 \defn \partialx \rho 
\end{equation}
obey the bound
\begin{equation}\label{3.56bis}
\Vert \zeta_1-h^{-1}\Vert_{C^0([-1,0], H^{s-\mez})}+
\lA \zeta_2\rA_{C^0([-1,0], H^{s-\mez})}\le \mathcal{F}(\| \eta \|_{H^{s+\mez}}).
\end{equation}
 Recall that $v=\widetilde \phi$ satisfies
\begin{equation}\label{eq:v}
(\partial_z^2  +\alpha\Delta_x  + \beta \cdot\nabla_x\partial_z   - \gamma \partial_z )v =0,\qquad
v\arrowvert_{z=0}=\phi\arrowvert_{y=\eta(x)}=
f,
\end{equation}
where
\begin{equation}\label{alpha}
\alpha\defn \frac{(\partial_z\rho)^2}{1+|\nabla_x   \rho |^2},\quad
\beta\defn  -2 \frac{\partial_z \rho \nabla_x \rho}{1+|\nabla_x  \rho |^2} ,\quad
\gamma \defn \frac{1}{\partial_z\rho}\bigl(  \partial_z^2 \rho
+\alpha\Delta_x \rho + \beta \cdot \nabla_x \partial_z \rho\bigr)
\end{equation}
together with the following bound
\begin{equation}
\lA \alpha - h^2\rA_{X^{s-\mez}([-1, 0])}+
\lA \beta\rA_{X^{s-\mez}([-1, 0])}+\lA \gamma\rA_{X^{s-\tdm}([-1, 0])}\le \mathcal{F}(\| \eta \|_{H^{s+\mez}}).
\label{regv}
\end{equation}
\subsubsection{Elliptic estimates}
The regularity of a function $v$ satisfying the elliptic problem \eqref{eq:v} is given in Proposition $3.16$, \cite{ABZ} which we also recall here for the sake of completeness. For $I\subset \xR$, define the interpolation spaces
\begin{equation}\begin{aligned}
X^\mu(I)&=C^0_z(I, H^{\mu}({\mathbf{R}}^d))\cap L^2_z(I, H^{\mu+\mez}({\mathbf{R}}^d)),\\
Y^\mu(I)&=L^1_z(I, H^{\mu}({\mathbf{R}}^d))+L^2_z(I, H^{\mu-\mez}({\mathbf{R}}^d)).
\end{aligned}
\end{equation}
\begin{prop}[see \protect{\cite[Proposition~$3.16$ and $(3.56)$]{ABZ}}]\label{ellip:regu}
Let~$d\ge 1$ and
$$
s> \mez + \frac d 2,\quad - \frac 1 2 \leq \sigma \leq s - \frac 1 2.
$$
1. Let $v$ be a solution to
\begin{equation}\label{eq:elliptic}
(\partial_z^2  +\alpha\Delta_x  + \beta \cdot\nabla_x\partial_z   - \gamma \partial_z )v =F_0,\qquad
v\arrowvert_{z=0}=f.
\end{equation}
Assume that $f\in H^{\sigma+1}({\mathbf{R}}^d)$,
$F_0\in Y^\sigma([-1,0])$ and
\begin{equation}\label{assu:var}
\lA \nabla_{x,z} v\rA_{X^{-\mez}([-1,0])}
<+\infty.
\end{equation}
Then for any~$z_0,~z_1\in (-1,0)$, $z_0>z_1$, we have $\nabla_{x,z}v \in X^{\sigma}([z_0,0])$ and
$$
\lA \nabla_{x,z} v\rA_{X^{\sigma}([z_0,0])}
\le \mathcal{F}(\| \eta \|_{H^{s+\mez}})\left\{\lA f\rA_{H^{\sigma+1}}+\lA F_0\rA_{Y^\sigma([z_1,0])}
+ \lA \nabla_{x,z} v\rA_{X^{-\mez}([z_1,0])} \right\},
$$
where $\cF$ depends only on~$\sigma$ and $z_0, z_1$. \\
2. Define the symbol
\begin{equation}\label{defi:A}
 A  = \frac{1}{2}\Big(-i \beta\cdot \xi
+   \sqrt{ 4\alpha \la \xi \ra^2 -  (\beta \cdot \xi)^2}\Big).
\end{equation}
Let 
\[
 0<\eps\le \mez,\quad \eps<s-\mez-\frac{d}{2}.
\] 
Then for any $\sigma\in [-\mez, s-\mez-\eps]$ and $-1<z_1<z_0<0$, $w:=(\partial_z-T_A)v$ satisfies the estimate
\begin{equation}\label{gain:w}
\Vert w\Vert_{X^{\sigma+\eps}([z_0, 0])}\le \mathcal{F}(\| \eta \|_{H^{s+\mez}})\left\{\lA f\rA_{H^{\sigma+1}}+\lA F_0\rA_{Y^{\sigma+\eps}([z_1,0])}
+ \lA \nabla_{x,z} v\rA_{X^{-\mez}([z_1,0])} \right\}.
\end{equation}
\end{prop}
\subsubsection{Paralinearization}
The principal symbol of the Diriclet-Neumann operator is given by
\begin{equation}\label{eq.lambda}
\lambda(x,\xi)\defn\sqrt{(1+\la\partialx\eta(x)\ra^2)\la\xi\ra^2-(\partialx\eta(x)\cdot\xi)^2}.
\end{equation}
With respect to this symbol, we consider the remainder
$$
R(\eta):=G(\eta)-T_\lambda.
$$
 Let us fix an index 
\begin{equation}\label{fix:index}
s>1+\frac{d}{2}
\end{equation}
for the rest of this paper, unless stated otherwise. It was proved in \cite{ABZ} (see Proposition 3.13) that 
\begin{equation}\label{estR:ABZ}
\lA R(\eta)f\rA_{H^{s-\mez}}\le \mathcal{F} \bigl(\| \eta \|_{H^{s+\mez}}\bigr)\lA f\rA_{H^{s}}.
\end{equation}
The preceding estimate is linear with respect to $f$. We prove in the next proposition that in the high frequency regime,  $R(\eta)f$ is "almost linear" with respect to both $\eta$ and $f$. %Remark that from now till section \ref{section:conclude}, we shall drop the lower index $n$ noticing that all the estimates we shall derive contain a dependence constant which can be uniformly bounded in $n$ (due to the fact that $\mathcal{M}^{(n)}_s(T)$ is uniformly bounded  in $n$).
\begin{prop}\label{DNest}
If $f\in H^s$ then
\[
\lA K_{\eps}R(\eta)f\rA_{H^{s-\mez}}\le \cF\big(\lA (\eta, f\big)\rA_{H^{s+\mez}\times H^s})\Big(o(1)_{\eps\to 0}+\lA (K_{\eps}\eta,  K_{\eps}f)\rA_{H^{s+\mez}\times H^s}\Big).
\]
Recall that $o(1)_{\eps\to 0}$ denotes some constant of the form $\eps^\kappa$ with $\kappa>0$ independent of the solution.
\end{prop}
\begin{proof}
For the sake of simplification, we shall write in this proof $A\les B$ if there exists a nondecreasing function $\cF:\xR^+\times \xR^+\to \xR^+$ such that  
\begin{equation}\label{nota:less}
A\le   \cF\big(\lA (\eta, f)\rA_{H^{s+\mez}\times H^s}\big)B.
\end{equation}
Going back to the proof of Proposition 3.13, \cite{ABZ} we have that $R$ is given by $R=R_1+R_2+R_3$,
where (recall that $v(x,z)=\phi(x, \rho(x,z))$ is a solution to the boundary value problem \eqref{eq:elliptic})
\begin{align*}
&R_1=(\zeta_1\partial_zv-T_{\zeta_1}\partial_zv)\mid_{z=0}-(\zeta_2\nabla v-T_{\zeta_2}\nabla v)\mid_{z=0},\\
&R_2=(T_{\zeta_1}\partial_zv-T_{\zeta_1}T_Av)\mid_{z=0},\\
&R_3=(T_{\zeta_1}T_A-T_{\zeta_1A})v\mid_{z=0}
\end{align*}
with $A\in \Gamma^1_{\frac{1}{2}+\rho_0}(\xR^d\times[-1, 0])$  defined by \eqref{defi:A} and $\zeta_1, \zeta_2$ are given by \eqref{zeta12} and $v$ is a solution to \eqref{eq:v} with $v(0)=f$.\\
$(i)$ {\it Estimate for $K_{\eps}R_3$}\\
Since $\zeta_1(0)\in H^{s-\mez}\subset W^{\mez+\rho_0,\infty}$, by Theorem \ref{theo:sc} $(ii)$ $T_{\zeta_1}T_A-T_{\zeta_1A}$ is of order $0+1-(\mez+\rho_0)=\frac{1}{2}-\rho_0$. Hence,
\begin{equation}\label{R1DN}
\lA K_{\eps}R_3\rA_{H^{s-\mez}}\le \eps^{\rho_0}\lA R_3\rA_{H^{s-\mez+\rho_0}}\lesssim \eps^{\rho_0}\lA v(0)\rA_{H^s}\les o(1)_{\eps\to 0}.
\end{equation}
$(ii)$ {\it Estimate for $K_{\eps}R_2$}\\
We write
\begin{align*}
K_{\eps}R_2&=K_{\eps}T_{\zeta_1}(\partial_zv-T_Av)\\
&=[K_{\eps}, T_{\zeta_1}](\partial_zv-T_Av)+T_{\zeta_1}K_{\eps}(\partial_zv-T_Av)\\
&=[K_{\eps}, T_{\zeta_1}](\partial_zv-T_Av)+T_{\zeta_1}(\partial_zK_{\eps}v-T_AK_{\eps}v-[K_{\eps},T_A]v)
\end{align*}
where all the values are evaluated at $z=0$.\\
 Using estimate \eqref{gain:w} with $\sigma_0=s-1,~\eps=\mez$, we have
\[
\lA \partial_zv-T_Av\rA_{H^{s-\mez}}\lesssim \lA f\rA_{H^s}\lesssim 1.
\]
On the other hand $\zeta_1(0)\in \Gamma^0_{\frac{1}{2}}$  and it follows from Lemma \ref{Keps} $(iii)$ that
\[
\lA [K_{\eps}, T_{\zeta_1}](\partial_zv-T_Av)\rA_{H^{s-\mez}}\les o(1)_{\eps\to 0}.
\]
Since $A\in \Gamma^1_{\mez+\rho_0}$ and $1-(\mez+\rho_0)<\mez$ we deduce again from Lemma \ref{Keps} $(iii)$  (note that $v(0)=f\in H^s$)
\[
\lA [K_{\eps}, T_A]v\rA_{H^{s-\mez}}\les o(1)_{\eps\to 0}.
\]
Finally, remark that $K_{\eps}v$ satisfies a similar elliptic equation as the one for $v$ (but inhomogeneous) with boundary value  $K_\eps v\arrowvert_{z=0}=K_{\eps}f$ . We claim that
\begin{equation}\label{eq:Kepsv}
\lA \partial_zK_{\eps}v-T_AK_{\eps}v\rA_{H^{s-\mez}}\lesssim o(1)_{\eps\to 0}+\lA (K_{\eps}\eta,  K_{\eps}f)\rA_{H^{s+\mez}\times H^s}
\end{equation}
which shall be proved in Lemma \ref{Kepsv}.\\
Putting together all the estimates gives
\begin{equation}
\lA K_{\eps}R_2\rA_{H^{s-\mez}}\lesssim o(1)_{\eps\to 0}+\lA (K_{\eps}\eta,  K_{\eps}f)\rA_{H^{s+\mez}\times H^s}.
\end{equation}
$(iii)$ {\it Estimate for $K_{\eps}R_1$}\\
The estimates for two terms in $R_1$ are similar but the first one is more difficult and we present the proof here. According to the Bony decomposition
\begin{align}\label{DN:1}
(\zeta_1\partial_zv-T_{\zeta_1}\partial_zv)&=
((\zeta_1-h^{-1})\partial_zv-T_{(\zeta_1-h^{-1})}\partial_zv)+h^{-1}\partial_zv-T_{h^{-1}}\partial_zv\\ \nonumber
&=T_{\partial_zv}(\zeta_1-h^{-1})+R((\zeta_1-h^{-1}), \partial_zv)+(h^{-1}-T_{h^{-1}})\partial_zv,
\end{align}
where all the values are estimated at $z=0$.\\
Applying \eqref{Bony} with $\alpha=s-1,~\beta=s-\mez$ with the remark that $\alpha+\beta-\frac{d}{2}>s-\mez+\rho_0$ yields at $z=0$
\begin{align*}
\lA K_{\eps}R(\partial_zv, \zeta_1-h^{-1}))\rA_{H^{s-\mez}}&\le \eps^{\rho_0}\lA R(\partial_zv, \zeta_1-h^{-1}))\rA_{H^{s-\mez+\rho_0}}\\
&\le \eps^{\rho_0}\lA \partial_zv\rA_{H^{s-1}}\lA (\zeta_1-h^{-1})\rA_{H^{s-\mez}}\les o(1)_{\eps\to 0}.
\end{align*}
For the paraproduct term, we write at $z=0$
\[
K_{\eps}T_{\partial_zv}(\zeta_1-h^{-1})=[K_{\eps},T_{\partial_zv}](\zeta_1-h^{-1})+T_{\partial_zv}K_{\eps}(\zeta_1-h^{-1}).
\]
Lemma \ref{Keps} $(iii)$ implies that $\lA [K_{\eps},T_{\partial_zv}](\zeta_1-h^{-1})\rA_{H^{s-\mez}}\les o(1)_{\eps\to 0}$.\\
To estimate $T_{\partial_zv}K_{\eps}(\zeta_1-h^{-1})$ we compute
\[
\zeta_1\mid_{z=0}= \frac{1 +|\partialx \rho |^2}{\partial_z \rho}\mid_{z=0}=\frac{1+|\nabla \eta|^2}{\eta+\delta\langle D_x\rangle\eta-\exp(-\delta \langle D_x\rangle)\eta+h}.
\]

Set $G=\eta+\delta\langle D_x\rangle\eta-\exp(-\delta \langle D_x\rangle)\eta$ we have
\[
\zeta_1(0)-\frac{1}{h}=\frac{1}{h}|\nabla \eta|^2-\frac{G}{G+h}\frac{1}{h}|\nabla \eta|^2-\frac{G}{(G+h)h}.
\]
Applying Theorem \ref{paralin} with $r=s-\mez$ we deduce that
\[
|\nabla\eta|^2=2\sum_{j=1}^dT_{\partial_j\eta}\partial_j\eta+R_4
\]
with $R_4\in H^{2s-1-\frac{d}{2}}$. Commuting this equation with $K_{\eps}$ yields
\[
\lA K_{\eps}|\nabla\eta|^2\rA_{H^{s-\mez}}\les o(1)_{\eps\to 0}+\lA K_{\eps}\eta\rA_{H^{s+\mez}}.
\]
By the same argument, we obtain
\[
\lA K_{\eps}\frac{G}{G+h}\rA_{H^{s-\mez}}\les o(1)_{\eps\to 0}+\lA K_{\eps}\eta\rA_{H^{s+\mez}}.
\]
For the product term we use the Bony decomposition
\[
\frac{G}{G+h}|\nabla \eta|^2=T_{\frac{G}{G+h}}|\nabla \eta|^2+T_{|\nabla \eta|^2}\frac{G}{G+h}+R(|\nabla \eta|^2, \frac{G}{G+h}).
\]
Then, commuting with $K_{\eps}$ and arguing as above we derive
\[
\lA K_{\eps}\frac{G}{G+h}|\nabla \eta|^2\rA_{H^{s-\mez}}\les o(1)_{\eps\to 0}+\lA K_{\eps}\eta\rA_{H^{s+\mez}}.
\]
Therefore,
\[
\lA K_{\eps}T_{\partial_zv}(\zeta_1-h^{-1})\rA_{H^{s-\mez}}\les o(1)_{\eps\to 0}+\|K_{\eps}\eta\|_{H^{s+\mez}}.
\]
On the other hand, it is easy to see that
\[
\lA K_{\eps}(h^{-1}\partial_zv-T_{h^{-1}}\partial_zv)\rA_{H^{s-\mez}}\les o(1)_{\eps\to 0}.
\]
In view of \eqref{DN:1}, we conclude
\[
\|K_{\eps}R_1\|_{H^{s-\mez}}\les o(1)_{\eps\to 0}+\|K_{\eps}\eta\|_{H^{s+\mez}}.
\]
\end{proof}
To complete the proof of Proposition \ref{DNest}, we are left with the claim \eqref{eq:Kepsv}.
\begin{lemm}\label{Kepsv}
We have (in view of the notation \eqref{nota:less})
\[
\lA \partial_zK_{\eps}v-T_AK_{\eps}v\arrowvert_{z=0}\rA_{H^{s-\mez}}\lesssim o(1)_{\eps\to 0}+\lA (K_{\eps}\eta,  K_{\eps}f)\rA_{H^{s+\mez}\times H^s}.
\]
\end{lemm}
\begin{proof}
Commuting \eqref{eq:v} with $K_{\eps}$ yields
\begin{equation}\label{eq:v'}
(\partial_z^2  +\alpha\Delta_x  + \beta \cdot\nabla_x\partial_z   - \gamma \partial_z )K_{\eps}v =F_0,\qquad
K_{\eps}v\arrowvert_{z=0}=K_{\eps}f,
\end{equation}
with
\[
F_0=-[K_{\eps},\alpha]\Delta_xv-[K_{\eps},\beta]\nabla_x\partial_zv+[K_{\eps},\gamma]\partial_z v.
\]
It follows from \eqref{gain:w} with $\sigma=s-1,~\eps=\mez$  that for any $z_1\in (-1, 0)$ we have with $J=(z_1, 0)$
\begin{equation}\label{normal:1}
\lA \partial_zK_{\eps}v-T_AK_{\eps}v\arrowvert_{z=0}\rA_{H^{s-\mez}}
\lesssim \lA K_{\eps}V\rA_{H^s}+ \lA\nabla_{x,z}K_{\eps}v\rA_{X^{-\mez}(J)}+\lA F_0\rA_{Y^{s-\mez}(J)}.
\end{equation}
On the other hand,  Proposition \ref{ellip:regu} with $\sigma=0$ also gives
\begin{equation}\label{normal:2}
 \lA\nabla_{x,z}K_{\eps}v\rA_{X^{-\mez}(J)}\les \eps^{\mez} \lA\nabla_{x,z}v\rA_{X^0(J)}\les \eps^{\mez}\lA f\rA_{H^1}\les o(1)_{\eps\to 0}.
\end{equation}
In views of \eqref{normal:1} and \eqref{normal:2}, to end the proof it remains to estimate $\lA F_0\rA_{Y^{s-\mez}(J)}$. \\
\hk {\bf 1.} We write
\[
[K_{\eps},\alpha]=[K_{\eps},T_{\alpha}]+K_{\eps}(\alpha-T_{\alpha})-(\alpha-T_{\alpha})K_{\eps}.
\]
Next, using the paralinearization estimate $v)$, Theorem \ref{pproduct} and Lemma \ref{Keps} $(ii)$, we have since $s>1+\rho_0+\frac{d}{2}$
\begin{align*}
\lA K_{\eps}(\alpha-T_{\alpha})\Delta v\rA_{L^2(J, H^{s-1})}&\le \eps^{\rho_0}\lA(\alpha-T_{\alpha})\Delta v\rA_{L^2(J, H^{s-1+\rho_0})}\\
&\les \eps^{\rho_0}\left(1+\lA \alpha-h^2\rA_{L^{\infty}(J, H^{s-\mez})}\right)\lA \Delta v\rA_{L^2(J, H^{s-\frac{3}{2}})}\les o(1)_{\eps\to 0}.
\end{align*}
Similarly,
\begin{align*}
\lA(\alpha-T_{\alpha}) K_{\eps}\Delta v\rA_{L^2(J, H^{s-1})}&
\les\left(1+\lA \alpha-h^2\rA_{L^{\infty}(J, H^{s-\mez})}\right)\lA  K_{\eps}\Delta v\rA_{L^2(J, H^{s-\tdm-\rho_0})}\\
&\le \eps^{\rho_0}\left(1+\lA \alpha-h^2\rA_{L^{\infty}(J, H^{s-\mez})}\right)\lA \Delta v\rA_{L^2(J, H^{s-\frac{3}{2}})}\les o(1)_{\eps\to 0}.
\end{align*}
On the other hand, since $\alpha(z)\in \Gamma^0_{\mez+\rho_0}$ and $\Delta v\in L^2(J, H^{s-\frac{3}{2}})$, we deduce from Lemma \ref{Keps} $(iii)$ that
\[
\lA [K_{\eps},T_{\alpha}]\Delta v\rA_{L^2(J, H^{s-1})}\les o(1)_{\eps\to 0}.
\]
Therefore,
\[
\lA [K_{\eps},\alpha]\Delta v\rA_{Y^{s-\mez}(J)}\le \lA [K_{\eps},\alpha]\Delta v\rA_{L^2(J, H^{s-1})}\les o(1)_{\eps\to 0}.
\]
Because $\alpha$ and $\beta$ as well as $\Delta v$ and $\nabla\partial_zv$ have the same regularity, the argument above also proves that
\[
\lA [K_{\eps},\beta]\nabla_x\partial_zv\rA_{Y^{s-\mez}(J)}\les o(1)_{\eps\to 0}.
\]
\hk {\bf 2.} Finally, to estimate $[K_{\eps},\gamma]\partial_zv$ we write
\begin{align*}
[K_{\eps},\gamma]\partial_zv&=(K_{\eps}T_{\gamma}-T_{\gamma}K_{\eps})\partial_zv+K_{\eps}T_{\partial_zv}\gamma+K_{\eps}R(\partial_zv, \gamma)-(\gamma-T_{\gamma})K_{\eps}\partial_zv\\
&=[K_{\eps},T_{\gamma}]\partial_zv+[K_{\eps},T_{\partial_zv}]\gamma+T_{\partial_zv}K_{\eps}\gamma+K_{\eps}R(\partial_zv, \gamma)-(\gamma-T_{\gamma})K_{\eps}\partial_zv.
\end{align*}
Since $\gamma\in L^2(J, H^{s-1})$ and $\partial_zv\in L^{\infty}(J, H^{s-1})$  with $s-1>\rho_0+\frac{d}{2}$, we can apply Lemma \ref{Keps} $(iii)$ to have
\[
\lA [K_{\eps},T_{\gamma}]\partial_zv\rA_{L^2(J, H^{s-1})}+\lA [K_{\eps},T_{\partial_zv}]\gamma\rA_{L^2(J, H^{s-1})}\les o(1)_{\eps\to 0}.
\]
On the other hand, it follows from \eqref{Bony} and Lemma \ref{Keps} $(ii)$ that
\[
\lA K_{\eps}R(\partial_zv, \gamma)\rA_{L^2(J, H^{s-1})}\les o(1)_{\eps\to 0}.
\]
Next, combining Theorem \ref{pproduct} $v)$ and Lemma \ref{Keps} $(ii)$ yields
\begin{align*}
\lA (\gamma-T_{\gamma})K_{\eps}\partial_zv\rA_{L^2(J, H^{s-1})}&\les \lA \gamma\rA_{L^2(J, H^{s-1})}\lA K_{\eps}\partial_zv\rA_{L^{\infty}(J, H^{s-1-\rho_0})}\\
&\les \eps^{\rho_0}\lA \gamma\rA_{L^2(J, H^{s-1})}\lA \partial_zv\rA_{L^{\infty}(J, H^{s-1})}\les o(1)_{\eps\to 0}.
\end{align*}
Therefore, we obtain
\[
\lA [K_{\eps},\gamma]\rA_{L^2(J, H^{s-1})}\les o(1)_{\eps\to 0}+\lA T_{\partial_zv}K_{\eps}\gamma\rA_{L^2(J, H^{s-1})}\les o(1)_{\eps\to 0}+\lA K_{\eps}\gamma\rA_{L^2(J, H^{s-1})}.
\]
Now we remark that by a completely similar argument as in part $(iii)$ of the proof of Proposition \ref{DNest} it holds that
\[
\lA K_{\eps}\gamma\rA_{L^2(J, H^{s-1})}\les o(1)_{\eps\to 0}+\lA K_{\eps}\eta\rA_{H^{s+\mez}}.
\]
Putting all the estimates together, we end up with
\[
\lA [K_{\eps},\gamma]\partial_zv\rA_{L^2(J, H^{s-1})}\les o(1)_{\eps\to 0}+\lA K_{\eps}\eta\rA_{H^{s+\mez}}.
\]
From {\bf 1.} and {\bf 2.} we conclude the proof.
\end{proof}
\subsection{Paralinearization and symmetrization of the system}\label{sec:para}
It was shown in \cite{ABZ} that the Zakharov/Craig--Sulem system can be reformulated in terms of the free surface $\eta$ and the trace $(B, V)$ of the velocity filed on the free surface. Considering $(\eta, \psi)$ a solution to \eqref{ww} on the time interval $[0, T]$, to recall this formulation we set
\begin{equation}\label{defi:unknowns}
\zeta = \partialx \eta, \quad
\B=\partial_y\phi \eval ,  \quad
V=\nabla_x \phi\eval , \quad \ma=-\partial_y P \eval ,
\end{equation}
where recall that~$\phi$ is the velocity potential and the pressure~$P=P(t,x,y)$ is given by
\begin{equation}\label{defi:P}
-P=\partial_t \phi+\mez \la \nabla_{x,y}\phi\ra^2+gy.
\end{equation}
\begin{prop}[see \protect{\cite[Proposition~4.3]{ABZ}}]\label{prop:newS}
Let~$s>\mez +\frac{d}{2}$. We have
\begin{align}
(\partial_{t}+V\cdot\partialx)\B&=\ma-g,\label{eq:B}\\
(\partial_t+V\cdot\partialx)V+\ma\zeta&=0,\label{eq:V}\\
(\partial_{t}+V\cdot\partialx)\zeta&=G(\eta)V+ \zeta G(\eta)\B +\gamma,\label{eq:zeta}
\end{align}
where the remainder ~$\gamma=\gamma(\eta,\psi,V, B)$ satisfies 
\begin{equation}\label{p42.1}
\lA \gamma\rA_{H^{s-\mez}}\le \mathcal{F}(\lA (\eta,\psi,V, B)\rA_{H^{s+\mez}\times H^{s+\mez}\times H^s\times H^s}).
\end{equation}
\end{prop}
\begin{rema}
For later use, we recall how $\gamma$ is determined.  Suppose that $V=(V^i),~i=\overline{1,d}$ and let $\theta^i,\Phi$ be the variational solutions of the following problems
  \begin{align}
 \Delta_{x,y}\theta^i &=0 \text{ in } \Omega, \quad \theta_i\arrowvert_{y=\eta} = V^i,\quad \frac{\partial \theta_i}{\partial \nu}\arrowvert_ \Gamma = 0, \quad i= \overline{1,d}, \label{eq:theta}\\
 \Delta_{x,y}\theta^{d+1} &=0 \text{ in } \Omega, \quad \theta^{d+1}\arrowvert_{y=\eta} = B, \quad \frac{\partial \theta_{d+1}}{\partial \nu}\arrowvert_ \Gamma = 0, \label{eq:theta0}\\
   \Delta_{x,y}\Phi&=0 \text{ in } \Omega, \quad \Phi\arrowvert_{y=\eta} = \psi, \quad \frac{\partial \Phi}{\partial \nu}\arrowvert_ \Gamma = 0.
\end{align}
   Then, define
   \begin{align}
 \gamma^i &:= (\partial_y - \nabla_x \eta \cdot \nabla_x)H^i\arrowvert_{y=\eta}, \quad H^i =  \partial_i \Phi - \theta^i, \quad i=1, \ldots,d \label{Ri}\\
   \gamma^{d+1} &:= (\partial_y - \nabla_x \eta \cdot \nabla_x)H^{d+1}\arrowvert_{y=\eta}, \quad H^{d+1}=  \partial_y \Phi - \theta^{d+1}\label{Rd+1}
 \end{align}
and $$\gamma:=(\gamma^i)_{i=\overline{1,d}}+\zeta \gamma^{d+1}.$$
In the case~$\Gamma = \emptyset$, one can see, at least formally, that~$\gamma =0$. The remainder $\gamma$ thus comes from the topography. 
\end{rema}
\begin{nota}
Set 
\[
\mathcal{M}_s(T)\defn \sup_{\tau\in [0,T]}
\lA (\eta(\tau)), \psi(\tau),B(\tau),V(\tau)\rA_{H^{s+\mez}\times H^{s+\mez} \times H^s\times H^s}.
\]
From now on, we write $A\les B$ if there exists a constant $C>0$ depending on $\mathcal{M}_s(T)$ such that $A\le CB$.
\end{nota}
The following lemma gives us an estimate for $K_{\eps}\gamma$ in the $H^{s-\mez}$ norm.
\begin{lemm}
There holds for all $t\in [0, T]$ that
\[
\lA K_{\eps}\gamma(t)\rA_{H^{s-\mez}}\les o(1)_{\eps\to 0}+\lA K_{\eps}U(t)\rA_{Z^s}.
\]
\label{Keps:gamma}
\end{lemm}
\begin{proof}
Remark that the product term $\zeta \gamma^{d+1}$ can be handled by using the Bony decomposition. The estimates for $\gamma^i$ are completely similar for $i=1,...,d+1$. For simplicity in notation we shall drop the upper index $i$ and a fixed time $t\in [0, T]$ in this proof. According to Notation \ref{tilde}, we have
\begin{equation}\label{formula:gamma}
 \gamma= \widetilde\gamma \arrowvert_{z=0}= \Big(\frac{1+ \vert \nabla_x \rho \vert^2}{\partial_z \rho} \partial_z - \nabla_x \rho\cdot \nabla_x \Big) \widetilde{H}_j\arrowvert_{z=0}= (\zeta_1\partial_z  - \zeta_2\cdot \nabla_x) \widetilde{H}_j\arrowvert_{z=0},
 \end{equation}
where $\zeta_1,~\zeta_2$ are defined in \eqref{zeta12}.\\
{\it Step 1.} We prove in this step that $\nabla_{x,z} \widetilde{H}\arrowvert_{z=0}\in H^{s-\mez}$ and moreover,
\begin{equation}\label{gamma:1}
\lA\nabla_{x,z} \widetilde{H}\arrowvert_{z=0}\rA_{H^{s-\mez}}\les \|V\|_{H^s}+\|\psi\|_{H^{s+\mez}}.
\end{equation}
Since $\Delta H_j=0$ in $\Omega$ and $H_j=0$ on the free surface, applying the elliptic regularity Proposition \ref{ellip:regu} with  fixed $-1<z_1<z_0<0$,
 we deduce that
$$
\lA \nabla_{x,z}  \widetilde{H}\rA_{X^{s-\mez}([z_0,0])}
\le \mathcal{F}(\| \eta \|_{H^{s+\mez}})\lA \nabla_{x,z}  \widetilde{H}\rA_{X^{-\mez}([z_1,0])}.
$$
%where $X^\mu(I)\defn C^0_z(I;H^{\mu}({\mathbf{R}}^d))\cap L^2_z(I;H^{\mu+\mez}({\mathbf{R}}^d))$.
Clearly,
\[
\lA \nabla_{x,z}  \widetilde{H}\rA_{X^{-\mez}([z_1,0])}\le \lA \nabla_{x,z}  \widetilde{\partial_i\Phi}\rA_{X^{-\mez}([z_1,0])}+\lA \nabla_{x,z}  \widetilde{\theta^i}\rA_{X^{-\mez}([z_1,0])}.
\]
 By virtue of the (variational) estimate \eqref{est:vari}, we have the bound
\[
\lA \nabla_{x,z}  \widetilde{\theta^i}\rA_{X^{-\mez}([-1,0])}\les  \|V\|_{H^s}.
\]
Next, using the change rule we get 
\begin{equation}\label{tildefo}
\widetilde{\partial_i \Phi} = (\partial_i - \frac{\partial_i \rho}{\partial_z \rho}\partial_z) \widetilde{\Phi}.
\end{equation}
Hence, in view of Proposition \ref{ellip:regu}
\[
\lA \nabla_{x}  \widetilde{\partial_i\Phi}\rA_{X^{-\mez}([z_1,0])}\le \lA \widetilde{\partial_i\Phi}\rA_{X^{\mez}([z_1,0])}\les  \lA \nabla_{x,z}\widetilde{\Phi} \rA_{X^{s-\mez}([z_1,0])}\les  \lA \psi \rA_{H^{s+\mez}}.
\]
Finally, using \eqref{tildefo} and equation \eqref{eq:v} we deduce that
\[
\lA \partial_z \widetilde{\partial_i\Phi}\rA_{X^{-\mez}([z_1 ,0])}\les  \lA \psi \rA_{H^{s+\mez}}.
\]
We finish step 1.\\
{\it Step 2.} In this step, all the values are evaluated at $z=0$. Using \eqref{formula:gamma} and the Bony decomposition, we write with $\zeta_1'=\zeta_1-h^{-1}$
\[
\gamma=h^{-1}\partial_z \widetilde{H} +T_{\zeta'_1}\partial_z\widetilde{H}+T_{\partial_z\widetilde{H}}\zeta_1+R(\zeta'_1, \partial_z\widetilde{H})-T_{\zeta_2}\cdot\nabla_x\widetilde{H}-T_{\nabla_x\widetilde{H}}\cdot\zeta_2-R(\nabla_x\widetilde{H}, \zeta_2).
\]
Since $\zeta'_1,~ \zeta_2, ~\nabla_{x,z}\widetilde{H}\arrowvert_{z=0}\in H^{s-\mez}$ and $s-\mez+s-\mez-\frac{d}{2}>s-\mez$, we can apply \eqref{Bony} and Lemma \ref{Keps} $(ii)$ to get
\begin{equation}\label{gamma:2}
\lA K_{\eps}R(\zeta'_1, \partial_z\widetilde{H})\rA_{H^{s-\mez}}+\lA K_{\eps}R( \zeta_2, \nabla_x\widetilde{H})\rA_{H^{s-\mez}}\les o(1)_{\eps\to 0}.
\end{equation}
Using estimates in part $(iii)$ of the proof of Proposition \ref{DNest}, we obtain that
\begin{equation}\label{gamma:3}
\lA K_{\eps}T_{\partial_z\widetilde{H}}\zeta'_1\rA_{H^{s-\mez}}+\lA K_{\eps}T_{\nabla_x\widetilde{H}}\zeta_2\rA_{H^{s-\mez}}\les o(1)_{\eps\to 0}+\lA K_{\eps}U\rA_{Z^s}.
\end{equation}
Now we write
\[
K_{\eps}(T_{\zeta'_1}\partial_z\widetilde{H}-T_{\zeta_2}\nabla_x\widetilde{H})=[K_{\eps},T_{\zeta'_1}]\partial_z\widetilde{H}+T_{\zeta'_1}K_{\eps}\partial_z\widetilde{H}-[K_{\eps}, T_{\zeta_2}]\cdot\nabla_x\widetilde{H}-T_{\zeta_2}\cdot K_{\eps}\nabla_x\widetilde{H}.
\]
Lemma \ref{Keps} $(iii)$ gives
\[
\lA[K_{\eps},T_{\zeta'_1}]\partial_z\widetilde{H} \rA_{H^{s-\mez}}+\lA [K_{\eps}, T_{\zeta_2}]\cdot\nabla_x\widetilde{H}\rA_{H^{s-\mez}}\les o(1)_{\eps\to 0}.
\]
Consequently,
\begin{align}\label{gamma:4}
\lA K_{\eps}(T_{\zeta'_1}\partial_z\widetilde{H}-T_{\zeta_2}\cdot\nabla_x\widetilde{H})\rA_{H^{s-\mez}}&\les o(1)_{\eps\to 0}+\lA T_{\zeta'_1}K_{\eps}\partial_z\widetilde{H}\rA_{H^{s-\mez}}+\lA T_{\zeta_2}\cdot K_{\eps}\nabla_x\widetilde{H}\rA_{H^{s-\mez}}\\\nonumber
&\les o(1)_{\eps\to 0}+\lA K_{\eps}\partial_z\widetilde{H}\rA_{H^{s-\mez}}+\lA K_{\eps}\nabla_x\widetilde{H}\rA_{H^{s-\mez}}.
\end{align}
We remark that as in the proof of Lemma \ref{Kepsv}, $K_{\eps}\widetilde{H}$ and $\widetilde{H}$ satisfy almost the same equation (modulo a quantity bounded by $o(1)_{\eps\to 0}+\lA K_{\eps}U\rA_{Z^s}$) so we have as in \eqref{gamma:1}
\begin{equation}\label{gamma:5}
\lA K_{\eps}\nabla_{x,z} \widetilde{H}\rA_{H^{s-\mez}}\les o(1)_{\eps\to 0}+\lA K_{\eps}U\rA_{Z^s}.
\end{equation}
In view of \eqref{gamma:4}, one gets
\[
\lA K_{\eps}(T_{\zeta'_1}\partial_z\widetilde{H}-T_{\zeta_2}\cdot\nabla_x\widetilde{H})\rA_{H^{s-\mez}}\les o(1)_{\eps\to 0}+\lA K_{\eps}U\rA_{Z^s}.
\]
Putting together all estimates above we obtain the desired result.
\end{proof}
Introduce $\lDx{s}:=(I-\Delta)^{\frac{s}{2}}$ and
\begin{equation}\label{defiUszetas}
W_s \defn \langle D_x\rangle^s K_\eps^2 V+T_\zeta \lDx{s}K_\eps^2\B,\quad \zeta_s\defn K_\eps^2\lDx{s}\zeta.
\end{equation}
System \eqref{eq:B}-\eqref{eq:zeta} is paralinearized as follows (see Proposition 4.8, \cite{ABZ}).
\begin{prop}\label{prop:csystem2}
We have
\begin{align}
&(\partial_t+T_V\cdot\partialx)W_s+T_\ma\zeta_s=f_1,\label{premiere}\\
&(\partial_{t}+T_V\cdot\partialx)\zeta_s=T_\lambda W_s +f_2,\label{seconde}
\end{align}
 where, for each time~$t\in [0,T]$, there holds
\begin{equation}\label{p49:estf}
\lA (f_1(t),f_2(t))\rA_{L^2\times H^{-\mez}}
\les o(1)_{\eps\to 0}+\lA K^\eps U\rA_{Z^s}.
\end{equation}
\end{prop}
\begin{lemm} \label{para:lemm1}We have
\begin{equation}\label{parastep1a}
(\partial_t+T_V\cdot\partialx)V+T_{\ma}\zeta+T_\zeta (\partial_t+T_V\cdot\partialx)B=h_1
\end{equation}
with
\[
\lA K_{\eps}h_1(t)\rA_{H^s}
\les o(1)_{\eps\to 0}+\lA K_{\eps}U(t)\rA_{Z^s}.
\]
\end{lemm}
\begin{proof}
It follows from the proof of Lemma 4.9, \cite{ABZ} that
\[
h_1=\sum_{j=1}^d(T_{\partial_jV}V+R(\partial_jV, V_j))+R(a-g, \zeta)+g\zeta-T_g\zeta.
\]
Firstly, since $V_j\in H^s$ applying (\ref{Bony}) and Lemma \ref{Keps} $(ii)$ gives
\[
\lA K_{\eps}R(\partial_jV, V_j))\rA_{H^s}\les o(1)_{\eps\to 0}.
\]
Similarly, since $\zeta\in H^{s-\mez}$ and from Proposition 4.6 in \cite{ABZ} $a-g\in H^{s-\mez}(\xR^d)$ we deduce that
\[
\lA K_{\eps}R(a-g, \zeta)\rA_{H^s}\les o(1)_{\eps\to 0}.
\]
It is easy to see that
\[
\lA K_{\eps}(g\zeta-T_g\zeta)\rA_{H^s}\le \|K_{\eps}\eta\|_{H^{s+\mez}}.
\]
Finally, we write
\[
K_{\eps}T_{\partial_jV}V=[K_{\eps}, T_{\partial_jV}]V+T_{\partial_jV}K_{\eps}V
\]
and apply lemma \ref{Keps} $(iii)$ to derive that
\[
\lA K_{\eps}T_{\partial_jV}V\rA_{H^s}\les o(1)_{\eps\to 0}+\lA K_{\eps}V\rA_{H^s}.
\]
The lemma is proved.
\end{proof}
{\it Proof of \eqref{premiere}}\\
We commute ~\eqref{parastep1a} with~$L:=K_\eps^2\lDx{s}$  to obtain
\begin{equation}\label{paralemm2:1}
(\partial_t+T_V\cdot\partialx)LV+T_{\ma}L\zeta+T_\zeta L(\partial_t+T_V\cdot\partialx)B=Lh_1-k_1
\end{equation}
where
\[
k_1= \left[L, T_V\right]\nabla V+\left[L, T_\ma \right]\zeta+\left[L, T_{\zeta}\right](\partial_t+T_V\cdot\partialx)B.
\]
Applying Lemma \ref{lemma:K2} $(i)$ to $[L, T_V]\nabla V$ with $\alpha=1+\rho_0$ and $\mu=0$ we get
\begin{equation}\label{raralemm2:2}
\lA \left[L, T_V\right]\nabla V \rA_{L^2}\les o(1)_{\eps\to 0}+ \lA K_{\eps}V \rA_{H^{s}}.
\end{equation}
On the other hand, since $a\in C^{1/2+\rho_0}$ Lemma \ref{lemma:K2} $(ii)$ with $\beta=\mez,~\rho=\rho_0$  applied to $[L, T_{\ma}]\zeta$  yield
\begin{equation}\label{paralemm2:3}
\lA \left[L, T_\ma \right]\zeta  \rA_{L^2}\les o(1)_{\eps\to 0}.
\end{equation}
Now, remark that by Theorem \ref{pproduct} $v)$, $(T_V-V)\nabla B\in H^{s-\mez}$ and thus equation \eqref{eq:B} implies
\[
(\partial_t+T_V\cdot\partialx)B=(a-g)+(T_V-V)\nabla B\in H^{s-\mez}.
\]
Then since $\zeta\in H^{s-\mez}$ one gets similarly to \eqref{paralemm2:2} that
\[
\lA \left[ L, T_\zeta\right](\partial_t+T_V\cdot\partialx)B  \rA_{ L^2} \les o(1)_{\eps\to 0}.
\]
We thus obtain
\[
\lA k_1(t)\rA_{L^2} \les o(1)_{\eps\to 0}+\lA K_{\eps}U(t)\rA_{Z^s}.
\]
Now in \eqref{paralemm2:1} we commute $L$ with $\partial_t+T_V\cdot\nabla$ to get
\[
T_\zeta L(\partial_t+T_V\cdot\partialx)B=T_\zeta (\partial_t+T_V\cdot\partialx)LB+T_\zeta[L, T_V] \cdot\nabla B
\]
where the second term is estimated as in \eqref{paralemm2:2}. In the first term, one applies Lemma $2.15$, \cite{ABZ} and the identity $4.13$, \cite{ABZ} to get
\[
\lA [T_\zeta, \partial_t+T_V\cdot\partialx ]LB\rA_{L^2}\les \left(\lA \zeta\rA_{L^\infty}\lA V\rA_{C^{1+\rho_0}}+\lA \partial_t\zeta+V\nabla\zeta\rA_{L^\infty}\right)\lA LB\rA_{L^2}\les \lA K^2_\eps B\rA_{H^s}.
\]
Summing up, we have proved that
\begin{equation}\label{paralemm2:2}
(\partial_t+T_V\cdot\partialx) LV
+T_{\ma} L\zeta+(\partial_t+T_V\cdot\partialx)T_\zeta L B=h_2
\end{equation}
for some remainder~$h_2$ satisfying
\[ \lA h_2\rA_{L^2} \le  o(1)_{\eps\to 0}+\lA K_{\eps}U(t)\rA_{Z^s}.
\]
By definition \eqref{defiUszetas} we get \eqref{premiere}.\\
{\it Proof of \eqref{seconde}}\\
It follows from equation \eqref{eq:zeta} that
\begin{align}\label{writef2}
(\partial_{t}+T_V\cdot\partialx)\zeta &=G(\eta)V+ \zeta G(\eta)\B+ \gamma+(T_V-V)\cdot \nabla\zeta\\ \nonumber
&=G(\eta)V+ \zeta G(\eta)\B+ \gamma+T_{\nabla\zeta}\cdot V +\sum_{j=1}^d R(\partial_j\zeta,V_j).
\end{align}
The estimate for $\Vert K_\eps\gamma\Vert_{H^{s-\mez}}$ is provided by Lemma \ref{Keps:gamma}. Since $\partial_j\zeta\in H^{s-\frac{3}{2}}\subset C^{-\mez+\rho_0}_*$ we deduce from Lemma \ref{Keps} $(ii)$ and the paraproduct rule $(iii)$ of Theorem \ref{pproduct} that
\begin{equation}\label{f2:0}
\lA K_{\eps}T_{\nabla\zeta}\cdot V  \rA_{H^{s-\mez}}\les \eps^{\rho_0}\lA T_{\nabla\zeta}\cdot V \rA_{H^{s-\mez+\rho_0}}\les \eps^{\rho_0}\lA V \rA_{H^{s}}\les o(1)_{\eps\to 0}.
\end{equation}
Because $s-\frac{3}{2}+s-\frac{d}{2}>s-\mez$, we also have from \eqref{Bony} and Lemma \ref{Keps} $(ii)$ that
\[
\lA K_{\eps}R(\partial_j\zeta,V_j)\rA_{H^{s-\mez}}\les o(1)_{\eps\to 0}.
\]
Next, using the paralinearization of the Dirichlet-Neumann operator and the Bony's decomposition, we get
\begin{equation}\label{f2:1}
G(\eta)V+ \zeta G(\eta)\B = T_\lambda W  +F_2(\eta,V,B),
\end{equation}
with
\[
W:=V+T_{\zeta}B
\]
and
\begin{align*}\label{sy:F1}
F_2&=[T_\zeta,T_\lambda]B+R(\eta)V+\zeta R(\eta)B+(\zeta-T_\zeta)T_\lambda B\\
&=[T_\zeta,T_\lambda]B+R(\eta)V+T_{\zeta} R(\eta)B+T_{R(\eta)B}\zeta+R(\zeta, R(\eta)B)+T_{T_{\lambda B}}\zeta+R(T_{\lambda}B, \zeta).
\end{align*}
\hk $(i)$ The commutator~$ [T_\zeta,T_\lambda]$ is of order $0+1-(\mez+\rho_0)=\mez-\rho_0$. Thus, again by Lemma \ref{Keps} $(ii)$,
$$
\lA K_{\eps}[T_\zeta,T_\lambda]B\rA_{H^{s-\mez}}\les \eps^{\rho_0}\lA [T_\zeta,T_\lambda]B\rA_{H^{s-\mez+\rho_0}}\les \eps^{\rho_0}\lA B\rA_{H^{s}}\les o(1)_{\eps\to 0}.
$$
\hk $(ii)$ Writing $K_{\eps}T_{\zeta} R(\eta)B=[K_{\eps}, T_{\zeta}] R(\eta)B+T_{\zeta}K_{\eps} R(\eta)B$ then using Lemma \ref{Keps} $(iii)$ and Proposition \ref{DNest} for estimating the remainder in paralinearization of the Dirichlet-Neumann operator, we deduce that
\begin{align*}
\lA K_{\eps}R(\eta)V\rA_{H^{s-\mez}}+\lA K_{\eps}T_{\zeta} R(\eta)B\rA_{H^{s-\mez}}&\les o(1)_{\eps\to 0}+\lA K_{\eps}R(\eta)V\rA_{H^{s-\mez}}+\lA K_{\eps}R(\eta)B\rA_{H^{s-\mez}}\\
&\les o(1)_{\eps\to 0}+\lA K_{\eps}U\rA_{Z^s}.
\end{align*}
\hk  $(iii)$ Similarly, we also have
\[
\lA K_{\eps}T_{R(\eta)B}\zeta\rA_{H^{s-\mez}}+\lA K_{\eps}T_{T_\lambda B}\zeta\rA_{H^{s-\mez}}\les o(1)_{\eps\to 0}+\lA K_{\eps}U\rA_{Z^s}
\]
and
\[
\lA K_{\eps}R(\zeta, R(\eta)B)\rA_{H^{s-\mez}}+\lA K_{\eps}R(T_{\lambda}B, \zeta_j)\rA_{H^{s-\mez}}\les o(1)_{\eps\to 0}.
\]
Putting together the estimates in $(i)$, $(ii)$ and $(iii)$ we end up with
\begin{equation}\label{f2:2}
\lA K_{\eps}F_2\rA_{H^{s-\mez}}=T_{\lambda}W+k_2,
\end{equation}
with
\[
\lA K_{\eps}k_2\rA_{H^{s-\mez}}\les o(1)_{\eps\to 0}+\lA K_{\eps}U\rA_{Z^s}.
\]
Now using Lemma \ref{Keps:gamma} and \eqref{writef2}-\eqref{f2:2} we obtain
\begin{equation}\label{f2:4}
(\partial_{t}+T_V\cdot\partialx)\zeta=T_{\lambda}W+h_2
\end{equation}
with
\[
\lA K_{\eps}h_2\rA_{H^{s-\mez}}\les o(1)_{\eps\to 0}+\lA K_{\eps}U\rA_{Z^s}.
\]
Next, commuting \eqref{f2:4} with $K^2_\eps\langle D_x\rangle^s$  gives
\[
(\partial_{t}+T_V\cdot\partialx)\zeta_s=T_{\lambda}K^2_\eps\langle D_x\rangle^sW+[K^2_\eps\langle D_x\rangle^s, T_{\lambda}]W+K^2_\eps\langle D_x\rangle^sh_2-[K^2_\eps\langle D_x\rangle^s, T_V]\nabla\zeta.
\]
The two commutators on the right-hand side are bounded in $H^{-1/2}$ by virtue of Lemma \ref{lemma:K2}. Indeed, Lemma \ref{lemma:K2} $(ii)$ applied with $\mu=-\mez,~m=1,~\beta=\mez,~\rho=\rho_0$ yields
 \[
\lA [K^2_\eps\langle D_x\rangle^s, T_{\lambda}]W\rA_{H^{-\mez}}\les o(1)_{\eps\to 0}.
\]
On the other hand, Lemma \ref{lemma:K2} $(i)$ applied with $\mu=-\mez,~m=0,~\alpha=\rho_0$ gives
 \begin{equation}\label{technical:K2}
\lA [K^2_\eps\langle D_x\rangle^s, T_V]\nabla\zeta\rA_{H^{-\mez}}\les o(1)_{\eps\to 0}+\lA K_\eps \nabla\zeta\rA_{H^{s-\tdm}}\les o(1)_{\eps\to 0}+\lA K_\eps \eta\rA_{H^{s+\mez}}.
\end{equation}
This is the only technical point where we need Lemma \ref{lemma:K2} $(i)$, which requires $K^2_\eps$ and not only $K_\eps$.\\
Finally, we write
\[
K^2_\eps\langle D_x\rangle^sW=W_s+[K^2_\eps\langle D_x\rangle^s, T_\zeta]B
\]
and conclude by Lemma \ref{lemma:K2} $(ii)$ (applied with $\mu=\mez,~m=0,~\beta=\mez,~\rho=\rho_0$) that
\[
\lA T_\lambda [K^2_\eps\langle D_x\rangle^s, T_\zeta]B\rA_{H^{-\mez}}\les \lA [K^2_\eps\langle D_x\rangle^s, T_\zeta]B\rA_{H^{\mez}}\les o(1)_{\eps\to 0}.
\]
The proof of \eqref{seconde} is complete.\\
\hk Using the fact that the Taylor coefficient~$\ma$ is a positive function, the authors performed in \cite{ABZ} a
symmetrization of the non-diagonal part of the system~\eqref{premiere}--\eqref{seconde}. Given Proposition \ref{prop:csystem2}, one can repeat exactly line by line the proof of Proposition $4.10$, \cite{ABZ} to derive the following result.
\begin{prop}\label{prop:sym}
Set
$$
\gamma=\sqrt{\ma \lambda},\quad q= \sqrt{\frac{\ma}{\lambda}},
$$
and set~$\theta_s=T_{q} \zeta_s$. Then it holds that
\begin{align}
&\partial_{t} W_s+T_{V}\cdot\partialx  W_s  + T_\gamma \theta_s = F_1,\label{systreduit1}
\\
&\partial_{t}\theta_s+T_{V}\cdot\partialx \theta_s - T_\gamma W_s =F_2,
\label{systreduit2}
\end{align}
for some source terms~$F_1,F_2$ satisfying
\[
\forall t\in [0, T],\quad \lA (F_1(t), F_2(t))\rA_{L^2\times L^2}\les o(1)_{\eps\to 0}+\lA K_{\eps}U(t)\rA_{Z^s}.
\]
\end{prop}
Now, multiplying \eqref{systreduit1} with $W_s$ and \eqref{systreduit2} with $\theta_s$, using the remark that
\begin{align*}
&\lA (T_{V(t)}\cdot \partialx)^* + T_{V(t)}\cdot \partialx \rA_{L^2\rightarrow L^2}
\le C\lA V(t)\rA_{W^{1,\infty}},\\
&\lA T_{\gamma(t)} -(T_{\gamma(t)})^*\rA_{L^2\rightarrow L^2}
\le CM^{1/2}_{1/2}(\gamma(t)),
\end{align*}
one derives the~$L^2$ estimate for~$(W_s, \theta_s)$. More precisely,
\begin{prop}\label{L2est}
We have for all $t\in [0, T]$ that
\begin{equation}\label{L^2est}
\begin{aligned}
&\frac{d}{dt}\lA (W_s(t), \theta_s(t))\rA_{L^2\times L^2}^2\\
&\les \left( o(1)_{\eps\to 0}+\lA K_{\eps}U(t)\rA_{Z^s}+\lA (W_s(t), \theta_s(t))\rA_{L^2\times L^2}\right) \lA (W_s(t), \theta_s(t))\rA_{L^2\times L^2}.
\end{aligned}
\end{equation}
\end{prop}
\subsection{From the new unknown to the original unknown}
Our purpose in this section is to prove the following estimate between the original unknown and the symmetrized unknown:
\begin{equation}\label{main2'}
\lA K^2_{\eps}U(t)\rA_{Z^s}\les o(1)_{\eps\to 0}+\lA (W_s(t), \theta_s(t))\rA_{L^2\times L^2},\quad \forall t\in [0, T].
\end{equation}
Recall that the functions~$W_s$ and~$\theta_s$ are obtained
from~$(\eta,V,\B)$ by
\[
W_s \defn K^2_\eps\lDx{s} V+T_\zeta K^2_\eps\lDx{s}\B,\quad \theta_s\defn T_qK^2_\eps\lDx{s}\nabla \eta=T_q\zeta_s.
\]
The fact that one can achieve \eqref{main2'} partly bases on the ellipticity of the symbol $q=\sqrt{\frac{a}{\lambda}}$, which in turn stems from the positivity of the Taylor coefficient $a$.
\begin{lemm}\label{eta}
There holds for all $t\in [0, T]$ that
\begin{equation}\label{ecBVBV}
\| K^2_{\eps}\eta(t) \|_{H^{s+\mez}}\les o(1)_{\eps\to 0}+\lA  \theta_s(t)\rA_{L^2}.
\end{equation}
\end{lemm}
\begin{proof}
Since $q\in \Gamma^{-\mez}_{\mez+\rho_0}$ and $\frac{1}{q}\in \Gamma^{\mez}_{\mez+\rho_0}$, by Theorem \ref{theo:sc} we have
\[
\zeta_s=T_{q^{-1}}T_q\zeta_s+R\zeta_s=T_{q^{-1}}\theta_s+R\zeta_s
\]
with $R$ is of order $-\mez-\rho_0$. Hence $\| T_{q^{-1}}\theta_s\|_{H^{-\mez}}\les \| \theta_s\|_{L^2}$ and
\[
\lA R\zeta_s\rA_{H^{-\mez}}\le \lA \zeta_s\rA_{H^{-1-\rho_0}}=\lA K^2_\eps \nabla\eta\rA_{H^{s-1-\rho_0}}\les o(1)_{\eps\to 0},
\]
Consequently,
\[
\lA K^2_\eps \nabla \eta\rA_{H^{s-\mez}}\les o(1)_{\eps\to 0}+\lA  \theta_s(t)\rA_{L^2},
\]
which combining with the fact that $\| K^2_\eps \eta\|_{H^{s-1/2}}\les o(1)_{\eps\to 0}$ concludes the proof.
\end{proof}
\begin{lemm}\label{BV}
There holds for all $t\in [0, T]$ that
\begin{equation}\label{ecBVBV}
\| K^2_\eps B(t) \|_{H^s}+\| K^2_\eps V(t) \|_{H^s}\les o(1)_{\eps\to 0}+\lA W_s(t)\rA_{L^2}.
\end{equation}
\end{lemm}
\begin{proof}
Set $W=V+T_{\zeta}B$. From the proof of Lemma 4.15, \cite{ABZ} we have
$$
B=T_{\frac{1}{e}}  \cn W -T_{\frac{1}{e}} \gamma'+SB
$$
where
\begin{equation}\label{defi:S}
S\defn T_{\frac{1}{e}} \bigl( -T_{\cn \zeta}-R(\eta)\bigr)+\bigl(I-T_{\frac{1}{e}}T_e\bigr), \quad e\defn -\lambda+i\zeta.\xi
\end{equation}
and $\|\gamma'\|_{H^{s-\mez}}\les 1$.
It is easy to see that  $\| K^2_{\eps}T_{\frac{1}{e}}\gamma'\|_{H^s}= o(1)_{\eps\to 0}$  and (thanks to Lemma \ref{Keps} $(iii)$)
\[
\lA K^2_{\eps}T_{\frac{1}{e}}  \cn W\rA_{H^s}\les o(1)_{\eps\to 0} +\lA K^2_{\eps}W(t)\rA_{H^s}.
\]
On the other hand, since $\cn\zeta\in C^{-\mez}_*$ we can apply Theorem \ref{pproduct} $(iii)$ to have
\[
\lA K^2_{\eps}T_{\frac{1}{e}}T_{\cn \zeta}B\rA_{H^s}\le \eps^{\mez}\lA T_{\frac{1}{e}}T_{\cn \zeta}B\rA_{H^{s+\mez}}\les \eps^{\mez}\lA T_{\cn \zeta}B\rA_{H^{s-\mez}}\les \eps^{\mez}\lA B\rA_{H^{s}}\les o(1)_{\eps\to 0}.
\]
Similarly, because $R(\eta)B\in H^{s-\mez}$ (by \eqref{estR:ABZ}) we also have
\[
\lA K^2_{\eps}T_{\frac{1}{e}}R(\eta)B\rA_{H^s}\les o(1)_{\eps\to 0}.
\]
Finally, $\bigl(I-T_{\frac{1}{e}}T_e\bigr)$ is of order $-1+1-\mez=-\mez$, hence
\[
\lA K^2_{\eps}(I-T_{\frac{1}{e}}T_q)B\rA_{H^s}\les o(1)_{\eps\to 0}.
\]
Putting all together we deduce
\[
\lA K^2_{\eps}B\rA_{H^s}\les o(1)_{\eps\to 0} +\lA K^2_{\eps}W(t)\rA_{H^s}.
\]
On the other hand, $K_\eps^2\langle D_x\rangle^sW=W_s+[K_\eps^2\langle D_x\rangle^s,T_\zeta]B$ and it follows from Lemma \ref{lemma:K2} $(ii)$ that
\[
\lA [K_\eps^2\langle D_x\rangle^s,T_\zeta]B\rA_{L^2}\les o(1)_{\eps\to 0}.
\]
Thus, we have proved that
\[
\lA K^2_{\eps}B\rA_{H^s}\les o(1)_{\eps\to 0} +\lA W_s(t)\rA_{L^2}.
\]
The desired estimate for  $K^2_{\eps}V$ in $H^s$ then follows by writing
\[
K^2_{\eps}V=K^2_{\eps}W-[K^2_{\eps},T_{\zeta}]B -T_{\zeta}K^2_{\eps}B.
\]
\end{proof}
Finally, we prove the estimate for $K^2_{\eps}\psi$ in the following lemma.
\begin{lemm}\label{psi}
\begin{equation}\label{psif}
\| K^2_{\eps}\psi(t) \|_{H^{s+\mez}}\les o(1)_{\eps\to 0}+\lA \theta_s(t)\rA_{L^2}.
\end{equation}
\end{lemm}
\begin{proof}
Using the identity~$\nabla\psi=V+\B \nabla \eta$ we can write
\begin{equation}\label{psi1}
K^2_{\eps}\nabla\psi=K^2_{\eps}V+K^2_{\eps}\B \nabla \eta=K^2_{\eps}V+K^2_{\eps}T_B \nabla \eta+K^2_{\eps}T_{\nabla \eta}B+K^2_{\eps}R(B, \nabla \eta).
\end{equation}
Since $B, V\in H^{s}$ and $\nabla \eta\in H^{s-\mez}\subset L^{\infty}$ we have by Lemma \ref{Keps} $(ii)$ that
\[
\lA K^2_{\eps}V\rA_{H^{s-\mez}}+\lA K^2_{\eps}T_{\nabla \eta}B\rA_{H^{s-\mez}}+\lA K^2_{\eps}R(B, \nabla \eta)\rA_{H^{s-\mez}}\les o(1)_{\eps\to 0}.
\]
On the other hand, by Lemma \ref{eta}
\[
\lA K^2_{\eps}T_B \nabla \eta\rA_{H^{s-\mez}}\les o(1)_{\eps\to 0}+\lA K^2_{\eps}\nabla \eta\rA_{H^{s-\mez}}\les o(1)_{\eps\to 0}+\lA \theta_s(t)\rA_{L^2}.
\]
Thus,
\[
\lA K^2_{\eps}\nabla\psi\rA_{H^{s-\mez}}\les o(1)_{\eps\to 0}+\lA \theta_s(t)\rA_{L^2}
\]
and this implies \eqref{psif}.
\end{proof}
In conclusion, combining Lemmas \ref{eta}, \ref{BV} and \ref{psi} we obtain the estimate \eqref{main2'}.
\subsection{Concluding the proof}\label{section:conclude}
First, it is easy to see by virtue of Theorem \ref{theo:sc} $(i)$ that
\begin{equation}\label{cl2}
\lA (W_s(t),\theta_s(t))\rA_{L^2\times L^2}\les \lA K^2_{\eps}U(t)\rA_{Z^s},\quad \forall t\in [0, T].
\end{equation}
Then the $L^2$ estimate in Proposition \ref{L2est} implies
\[
\lA (W_s(t), \theta_s(t))\rA_{L^2\times L^2}\les o(1)_{\eps\to 0}+\lA K^2_{\eps}U(0)\rA_{Z^s}+\int_0^t\lA K_{\eps}U(\tau)\rA_{Z^s}d\tau.
\]
From \eqref{cl2} and the relation \eqref{main2'} proved in the previous section, we deduce that
\begin{equation}\label{cl3}
\lA K^2_{\eps}U(t)\rA_{Z^s}\les o(1)_{\eps\to 0}+\lA K^2_{\eps}U(0)\rA_{Z^s}+\int_0^t\lA K_{\eps}U(\tau)\rA_{Z^s}d\tau.
\end{equation}
Now, using the obvious inequality
\[
k_{\eps}=k_{\eps}-k_{\eps}^2+k_{\eps}^2=k_{\eps}(1-k_{\eps})+k_{\eps}^2=k_{\eps}\jmath_{\eps}+k_{\eps}^2\les k_{\eps}^2+\jmath_{\eps}^2
\]
we have
\[
\lA K_{\eps}(U_n(\tau)-U_0(\tau))\rA_{Z^s}\les \lA K^2_{\eps}(U_n(\tau)-U_0(\tau))\rA_{Z^s}+\lA J^2_{\eps}(U_n(\tau)-U_0(\tau)))\rA_{Z^s}.
\]
Hence in view of \eqref{cl3} with $U=U_n$, it holds that
\begin{align*}
\lA K^2_{\eps}U_n(t)\rA_{Z^s}&\les o(1)_{\eps\to 0}+\lA K^2_{\eps}U_n(0)\rA_{Z^s}+\int_0^t\lA K_{\eps}U_0(\tau)\rA_{Z^s}d\tau+\int_0^t\lA K^2_{\eps}U_n(\tau)\rA_{Z^s}d\tau\\
&+\int_0^t\lA K^2_{\eps}U_0(\tau)\rA_{Z^s}d\tau+\int_0^t\lA J^2_{\eps}(U_n(\tau)-U_0(\tau))\rA_{Z^s}d\tau,
\end{align*}
where we have used the fact that $\mathcal{M}^{(n)}_s(T)$ is bounded in $n$. The Gronwall inequality then gives us
\begin{align}\label{main:est}
\lA K^2_{\eps}U_n\rA_{C([0, T], Z^s)}&\les o(1)_{\eps\to 0}+\lA K^2_{\eps}U_n(0)\rA_{Z^s}+\int_0^T\lA K_{\eps}U_0(\tau)\rA_{Z^s}d\tau+\\\nonumber
&\quad+\int_0^T\lA K^2_{\eps}U_0(\tau)\rA_{Z^s}d\tau+\int_0^T\lA J^2_{\eps}(U_n(\tau)-U_0(\tau))\rA_{Z^s}d\tau.
\end{align}
Consequently,
\begin{multline*}
\lA K^2_{\eps}U_n\rA_{C^0([0, T], Z^s)}\les  o(1)_{\eps\to 0}+\lA K^2_{\eps}U_n(0)\rA_{Z^s}+\lA K_{\eps}U_0\rA_{C^0([0, T], Z^s)}\\
+\lA J^2_{\eps}(U_n-U_0)\rA_{C^0([0, T], Z^s)}.
\end{multline*}
Then in views of the estimates \eqref{est:main}, \eqref{eq.lipschitz}  we deduce
\begin{align}\label{est:final}
&\lA U_n-U_0\rA_{C^0([0, T], Z^s)}\\\nonumber
&\les \frac{C}{\eps}\lA (U_n-U_0)\mid_{t=0}\rA_{Z^{s-1}}+\lA K^2_{\eps}U_0\rA_{C^0([0, T], Z^s)}+o(1)_{\eps\to 0}+\lA K^2_{\eps}U_n(0)\rA_{Z^s}+\\\nonumber
&\quad+\lA K_{\eps}U_0\rA_{C^0([0, T], Z^s)}+\lA J^2_{\eps}(U_n-U_0)\rA_{C^0([0, T], Z^s)}\\\nonumber
&\les \frac{C}{\eps}\lA (U_n-U_0)\mid_{t=0}\rA_{Z^{s-1}}+\lA K_{\eps}U_0\rA_{C^0([0, T], Z^s)}+\lA K^2_{\eps}U_n(0)\rA_{Z^s}+o(1)_{\eps\to 0}.
\end{align}
Now, fix $\eps$ and let $n\to \infty$ in \eqref{est:final} and take into account the fact that  $U_n(0)$ converges to $U_0(0)$ in $Z^s$, we have
\begin{align*}
\limsup_{n\to\infty}\lA U_n-U_0\rA_{C([0, T], Z^s)}&\les  \lA K_{\eps}U_0\rA_{C^0([0, T], Z^s)}+o(1)_{\eps\to 0}.
\end{align*}
Notice that since $U_0\in C([0, T], Z^s)$, the set $\{U_0(t):t\in [0, T]\}$ is compact in $Z^s$ and thus it follows from Lemma \ref{Keps} $(i)$ that
\[
\lim_{\eps\to 0}\lA K_{\eps}U_0\rA_{C^0([0, T], Z^s)}=0.
\]
Therefore, letting $\eps\to 0$ we end up with
\[
\lim_{n\to\infty}\lA U_n-U_0\rA_{C^0([0, T], Z^s)}=0
\]
which means that the solution map is continuous in the strong topology of $Z^s$ and thus the proof of Theorem \ref{maintheo} is complete.
\section*{Acknowledgment}
This work was partially supported by the labex LMH through the grant no ANR-11-LABX-0056-LMH in the "Programme des Investissements d'Avenir"  and by Agence Nationale de la Recherche project  ANA\'E ANR-13-BS01-0010-03. I would like to sincerely thank Prof. Nicolas Burq for many fruitful discussions and constant encouragement during this work. I thank the referee for useful comments that helped improve the presentation of the manuscript.

\end{document}